\documentclass{article}
\usepackage{amsthm}
\usepackage{amsfonts}
\usepackage {amsmath}
\usepackage{mathrsfs}
\usepackage{graphicx} 
\usepackage{color}
\usepackage{rotating}

\def\Sp{{\rm Sp}}
\def\C{{\rm C}}

\def\tr{{\rm tr}}

\newtheorem{lemma}{Lemma}[section]
\newtheorem{theorem}[lemma]{Theorem}
\newtheorem{theoremABC}[lemma]{Theorem}
\newtheorem{corollary}[lemma]{Corollary}

\newtheorem{remark}[lemma]{Remark}

\title{Square Roots of Symplectic Matrices}
\author{Qinglong Zhou$^{1,2}$\thanks{Partially supported by NSFC (No.12171426), the Natural Science Foundation of Zhejiang Province (No. Y19A010072) and the Fundamental Research Funds for the Central Universities (No. 226-2024-00136).
		E-mail: zhouqinglong@zju.edu.cn. }\\
	$^{1}$ School of Mathematical Sciences,\\Zhejiang University, Hangzhou 310058, Zhejiang, China\\
    $^{2}$ Institut f\"ur Mathematics,\\Augsburg Universit\"at, 86159 Augsburg, Germany\\
}

\begin{document}

\maketitle

\begin{abstract}
    This paper characterizes the existence of real symplectic square roots for symplectic matrices. The decomposition of Wonenburger matrices with respect to their eigenvalues permits a partition of the problem into three primary cases. A symplectic matrix whose spectrum avoids the negative real axis is shown to always admit a real symplectic square root. For a negative hyperbolic matrix, such a root exists if and only if the half-dimension is even and the matrix itself is a $\diamond$-square. In the degenerate case of eigenvalue $-1$, a necessary and sufficient condition is established via a decomposition into specific standard blocks. 
\end{abstract}

\section{Introduction}

The problem of finding a square root for a matrix is a classical question in linear algebra with deep implications. When this problem is considered within the symplectic group $\Sp(2n)$, whose elements represent linear canonical transformations in Hamiltonian mechanics, it acquires a significant geometric and dynamical meaning. A prominent application arises in the study of symmetric periodic orbits in Hamiltonian systems. Specifically, the existence of a doubly symmetric periodic orbit often translates into the algebraic condition that the associated linearized Poincaré map (a symplectic matrix) must admit a square root that itself possesses specific symmetry properties (\cite{FrM1,FrM2, Bat, AyB}).

There is a substantial literature on the square root function of a matrix, many of which are listed in the References section (\cite{CrL, Cul, Par, Lev, HaY, LaT, Hig, Sch, Sul, Cho}).
For a general real matrix, the existence of a real square root is governed by a classical Jordan-theoretic criterion: the obstruction comes from negative real eigenvalues, and the Jordan blocks associated with each such eigenvalue must occur with suitable even multiplicities. Thus, at the level of general linear algebra, the square-root problem is essentially controlled by the real Jordan form. The symplectic setting considered in this paper is more rigid. Even when the general matrix-theoretic obstruction disappears, one must still require the square root to lie in the symplectic group. This additional constraint couples the Jordan structure with the symplectic form and with the block data in a Wonenburger representative. Consequently, the problem is no longer determined solely by the Jordan form of the matrix, but also by how the corresponding spectral subspaces are paired symplectically. The purpose of this paper is to make this extra structure explicit and to formulate the resulting necessary and sufficient conditions within the real symplectic group.

Since $\Sp(2)=SL_2(\mathbb{R})$,
 the square root problem for $2\times2$ symplectic matrices has been considered (\cite{Sul}), but the picture in higher dimensions is less complete. 
This paper focuses on the square root problem within the real symplectic group $\Sp(2n)$, whose elements are central to Hamiltonian mechanics and symplectic geometry. We establish precise necessary and sufficient conditions for a symplectic matrix 
$M\in\Sp(2n)$ to have a real symplectic square root, tackling the problem by analyzing the spectral properties and Jordan structure of $M$. 

We begin by defining the key concepts used throughout the paper.
Let $M\in M_n(\mathbb{R})$. Then $X\in M_n(\mathbb{R})$ is a {\bf (real) square root} of $M$, if $X^2 = M$.
Furthermore, if $M,X\in\Sp(2n)$, we call $X$ a {\bf (real) symplectic square root} of $M$.
If there is no confusion, for a symplectic matrix $M$,
the real symplectic square root of $M$ shall be denoted briefly by ``the square root of $M$".

A particularly useful representation is provided by Wonenburger matrices. A symplectic matrix is called a {\bf Wonenburger matrix} if it can be written in the block form
$$
M=M_{A,B,C}=\begin{pmatrix}A&B\\C&A^T\end{pmatrix}\in\Sp(2n),
$$
where
\begin{equation}\label{constraints.of.Wonenburger.form}
    B=B^T,\quad C=C^T,\quad AB=BA^T,\quad A^TC=CA,\quad A^2-BC=I,
\end{equation}
equations which ensure that $M$ is symplectic.
A different choice of basis amounts to acting with an invertible matrix $R\in{\rm GL}_n(\mathbb{R})$, via
$$
R_*(A,B,C)=(RAR^{-1},RBR^T,(R^T)^{-1}CR^{-1}),
$$
i.e., $M_{A,B,C}$ is replaced by $M_{R_*(A,B,C)}$.
Thus we also denote it by
$$
R_*\begin{pmatrix}A&B\\C&A^T\end{pmatrix}
=\begin{pmatrix}RAR^{-1}&RBR^T\\(R^T)^{-1}CR^{-1}&(RAR^{-1})^T\end{pmatrix}.
$$
We denote the space of Wonenburger matrices by
$$
\Sp^{\mathcal{I}}(2n)=\{M_{A,B,C}:\;A,B,C\;{\rm satisfy}\;(\ref{constraints.of.Wonenburger.form})\}.
$$
A fascinating theorem due to Wonenburger \cite{Won} states
that every element $M\in\Sp(2n)$ can be written as the product
of two linear antisymplectic involutions, i.e., $M=\mathcal{I}_1\mathcal{I}_2$.
This factorization leads to the following fundamental conjugacy result (\cite{FrM1}):
\begin{theorem}
    (Wonenburger) Every symplectic matrix $M\in\Sp(2n)$ is symplectically conjugate to a Wonenburger matrix.
\end{theorem}
By Theorem 1.1, it suffices to consider the real symplectic square root of a Wonenburger matrix.

To describe the decomposition of symplectic matrices, we use the 
{\bf $\diamond$-sum} introduced in \cite{Lon},
For two $2m_k\times 2m_k$ matrices of square block form
$M_k=\left(\begin{matrix}A_k&B_k\cr
                   C_k&D_k\end{matrix}\right)$ with $k=1, 2$,
the ``$\diamond$"-sum of $M_1$ and $M_2$ is defined (c.f. \cite{Lon} called product there) by
the following $2(m_1+m_2)\times 2(m_1+m_2)$ matrix $M_1\diamond M_2$:
$$
M_1\diamond M_2=\left(\begin{matrix}A_1 &   0 & B_1 &   0\cr
                             0   & A_2 &   0 & B_2\cr
                             C_1 &   0 & D_1 &   0\cr
                             0   & C_2 &   0 & D_2\end{matrix}\right),
$$
and $M^{\diamond k}$ denotes the $k$-fold  $\diamond$-sum of $M$.

Our main results are summarized in the following three theorems, which collectively address the total spectral cases. 
First,
if the symplectic matrix has no eigenvalues on the negative real axis $\mathbb{R}^-$, then it always admits a real symplectic square root.
 Indeed, we have

\renewcommand{\theoremABC}{\noindent{\bf Theorem A.}}
\begin{theoremABC}
\label{non-negative.hyperbolic.case}
{\it
    Let $M=\begin{pmatrix}A&B\\C&A^T\end{pmatrix}$ be a Wonenburger matrix with no eigenvalues on the negative real axis
    $\mathbb{R}^-$.
    In the Wonenburger representation, the condition $\sigma(M)\cap\mathbb{R}^-=\emptyset$ 
    is equivalent to $\sigma(A)\cap(-\infty,-1]=\emptyset$.
    Then $M$ has at least one real symplectic square root.
}
\end{theoremABC}
\medskip

This result demonstrates that for matrices whose spectrum avoids the negative real axis, the square root problem admits a constructive solution. 

If the symplectic matrix is negative hyperbolic, meaning its eigenvalues are negative real numbers (other than $-1$), we have

\renewcommand{\theoremABC}{\noindent{\bf Theorem B.}}
\begin{theoremABC}
\label{negative.hyperbolic.case}
{\it
    Let $M=\begin{pmatrix}A&B\\C&A^T\end{pmatrix}\in\Sp(2n)$ be a Wonenburger matrix whose spectrum is contained in $\{\lambda,{\lambda}^{-1}\}$,
    and suppose that both \(\lambda\) and \(\lambda^{-1}\) have
    multiplicity \(n\).
    Then $M$ has at least one real symplectic square root if and only if
    the half-dimension $n$ is even, and 
    $M\sim M_1^{\diamond2}$ for some lower-dimensional Wonenburger  matrix $M_1$.
}
\end{theoremABC}
\medskip
Here and hereafter, $\sim$ denotes symplectic similarity, i.e., similarity by a symplectic matrix.

Theorem B states that, in the negative hyperbolic case, the square-root problem is reduced to a parity condition together with a $\diamond$-doubling condition. 
The dimension constraint ($n$ even) reflects the fact that negative hyperbolic matrices introduce additional topological constraints that can only be resolved when the dimension has the appropriate parity.

The most intricate case involves the eigenvalue $-1$, where nilpotent structures in the Jordan normal form introduce significant complexity.

\renewcommand{\theoremABC}{\noindent{\bf Theorem C.}}
\begin{theoremABC}\label{-1.case}
{\it
    Let $M=\begin{pmatrix}A&B\\C&A^T\end{pmatrix}\in\Sp(2n)$ be a Wonenburger matrix whose only eigenvalue is $-1$.
    Then $M$ has a real symplectic square root
    if and only if there exists $P\in\Sp(2n)$ such that
\begin{equation}\label{existence.condition.of.eig.-1}
    PMP^{-1}=\left[\begin{pmatrix}J_{m_1}(-1)&s_1{K}_{m_1}\\{D}_{1}&J_{m_1}(-1)^T\end{pmatrix}\diamond\ldots\diamond\begin{pmatrix}J_{m_p}(-1)&s_p{K}_{m_p}\\{D}_{p}&J_{m_p}(-1)^T\end{pmatrix}\right]^{\diamond2}\diamond
    \left[\begin{pmatrix}A_1&B_1\\C_1&A_1^T\end{pmatrix}\diamond\ldots\diamond \begin{pmatrix}A_q&B_q\\C_q&A_q^T\end{pmatrix}\right],
\end{equation}
where $D_i, 1\le i\le p$ and $C_j,1\le j\le q$
are symmetric matrices satisfying the corresponding Wonenburger relations, and
\begin{equation}\label{Aj.Bj}
A_j=\begin{pmatrix}J_{l_j+1}(-1)&O\\O&J_{l_j}(-1)\end{pmatrix},\qquad
B_j=\begin{pmatrix}
                        & & t_jK_{l_j}\\
                        & 0& &\\
                        t_jK_{l_j}& & &
                    \end{pmatrix}
\end{equation}
for some $m_i\in\mathbb{N},s_i\in\{\pm1\},\;1\le i\le p$
and $l_j\in\mathbb{N}\cup\{0\},t_j\in\{\pm1\},\;1\le j\le q$.
Here $J_{\cdot}(\lambda)$ is the Jordan block with eigenvalue $\lambda$, and $K_{l_j}$ is the $l_j\times l_j$ matrix with anti-diagonal entries all equal to $1$ and the other entries all equal to $0$, i.e.,
\begin{equation}\label{K}
    K_{l_j}=\left(\begin{matrix}
        0 & 0 & \cdots & 0& 1\\
        0 & 0 &\cdots & 1 & 0\\
        \vdots& \vdots& \begin{sideways}$\ddots$\end{sideways}& \vdots& \vdots\\
        0& 1 & \cdots& 0& 0\\
        1& 0& \cdots& 0& 0
    \end{matrix}\right).
    \end{equation}
    }
\end{theoremABC}
\medskip

Theorem C provides a complete characterization for the degenerate case where $-1$ is the only eigenvalue. The result shows that a real symplectic square root exists if and only if the matrix is conjugate to a specific standard form built from identifiable blocks via the $\diamond$-sum. The presence of nilpotent structures in the Jordan blocks introduces significant complexity, which is resolved by the detailed normal form described in the theorem.

Combining the above three cases, we obtain the following complete criterion.

\renewcommand{\theoremABC}{\noindent{\bf Theorem D.}}
\begin{theoremABC}\label{main.theorem}
{\it For any $M\in\Sp(2n)$, $M$ has a real symplectic square root if and only if, after a symplectic conjugacy and the corresponding spectral $\diamond$-decomposition, it can be written in the form
\[
M \sim M_0
\diamond
\left(\mathop{\diamond}_{\lambda\in\Lambda} M_\lambda\right)
\diamond M_{-1},
\]
where the following conditions hold:

(i) The component $M_0$ has no eigenvalues on the negative real axis. 
This component imposes no obstruction.

(ii) The set $\Lambda$ consists of one representative $\lambda<-1$ from each reciprocal pair
$\{\lambda,\lambda^{-1}\}\subset \mathbb R^-\setminus\{-1\}$.
For each $\lambda\in\Lambda$, the component $M_\lambda$ has spectrum contained in
$\{\lambda,\lambda^{-1}\}$.
If $M_\lambda\in \Sp(2n_\lambda)$, then $n_\lambda$ is even and
\[
M_\lambda\sim N_\lambda^{\diamond 2}
\]
for some lower-dimensional Wonenburger matrix $N_\lambda$.

(iii) The component $M_{-1}$, whose spectrum is contained in $\{-1\}$, is symplectically conjugate to a matrix of the standard form described in Theorem C.
}
\end{theoremABC}

Here components which do not occur in the spectrum of $M$ are simply omitted. In particular, the only obstructions to the existence of a real symplectic square root come from the negative real spectrum: each negative hyperbolic component must satisfy the parity and $\diamond$-doubling condition in Theorem B, while the $-1$-component must satisfy the normal-form condition in Theorem C.

The paper is organized as follows. In Section \ref{sec:2} we first study the square-root problem in low dimensions. The cases $\Sp(4)$ and $\Sp(6)$ are treated in detail by using the Cayley--Hamilton theorem and the coefficients of the characteristic polynomial; the resulting criteria also serve as model computations for the general theory. 
In Section \ref{sec:3} we establish the decomposition results for Wonenburger matrices. These results allow us to reduce the square-root problem to symplectic spectral components and prepare the analysis of the degenerate $-1$-component. Finally, in Section \ref{sec:4} we prove the main theorems stated above: the unobstructed case with no negative real spectrum, the negative hyperbolic case, and the degenerate case in which the spectrum is concentrated at $-1$. Combining these cases gives the complete criterion for the existence of real symplectic square roots in $\Sp(2n)$.


\section{Real symplectic square roots in low dimensions}
\label{sec:2}

\subsection{Real symplectic square roots of $4\times4$ symplectic matrices}

The square-root problem in $\Sp(2)$ is relatively elementary,
and has also been studied in \cite{FrM1} and \cite{Lev}. 
Indeed, we have
\begin{theorem}\label{thm:sqrt.dim2}
    Let $M\in\Sp(2)$. A necessary condition for $M$ to have a real symplectic square root is
$\tr(M)\ge-2$.
Moreover, \(M\) has a real symplectic square root if and only if one of the following mutually exclusive cases holds: 

    (i) $\tr(M)>-2$;

    (ii) $M=-I_2$.
\end{theorem}

\begin{remark}
If $\tr(M)<-2$ (in which case $M$ has two distinct negative eigenvalues), or if $\tr(M)=-2$ (in which case $M$ has two identical eigenvalues equal to $-1$) and $M$ is symplectic conjugate to $\begin{pmatrix}-1&\pm1\\0&-1\end{pmatrix}$,
then $M$ admits no real symplectic square root.
\end{remark}

For later convenience, as introduced in \cite{Lon}, we will denote
\begin{eqnarray}
    R(\theta)&=&\begin{pmatrix}\cos\theta&-\sin\theta\\ \sin\theta&\cos\theta\end{pmatrix},\quad \theta\in(0,\pi)\cup(\pi,2\pi),\\
    D(\lambda)&=&\begin{pmatrix}\lambda&0\\0&\frac{1}{\lambda}\end{pmatrix},\quad \lambda\in\mathbb{R}\backslash\{0,\pm1\},\\
    N_1(-1,b)&=&\begin{pmatrix}-1&b\\0&-1\end{pmatrix},\quad b\in\mathbb{R}.
\end{eqnarray}

We now examine the square-root problem in dimension four, where the Cayley--Hamilton theorem gives an explicit relation between a square root $X$ and the given matrix $M$. This allows us to derive a concrete criterion in terms of the coefficients of the characteristic polynomial of $M$, with the degenerate cases treated separately.

Suppose $M,X\in\Sp(4)$ and $X^2=M$. Now the Cayley-Hamilton Theorem states that
\begin{equation}\label{C-H4}
X^4-\tr(X)X^3+(\sigma_2+2)X^2-\tr(X)X+I=0.
\end{equation}
Here $\sigma_2=(x_1+\frac{1}{x_1})(x_2+\frac{1}{x_2})$ where $\sigma(X)=\{x_1,\frac{1}{x_1},x_2,\frac{1}{x_2}\}$.
Plugging $X^2=M$ into (\ref{C-H4}), we obtain
\begin{equation}\label{X.eq1}
(\tr X)X(M+I)=M^2+(\sigma_2+2)M+I.
\end{equation}

Suppose the characteristic polynomial of $M$ is given by
$$
p_{M}(t)=t^4-a_1t^3+a_2t^2-a_1t+1,
$$
where 
\begin{eqnarray}
    a_1&=&\tr M=x_1^2+\frac{1}{x_1^2}
             +x_2^2+\frac{1}{x_2^2}
    \label{a1.dim4}\\
    a_2&=&(x_1^2+\frac{1}{x_1^2})(x_2^2+\frac{1}{x_2^2})+2.
    \label{a2.dim4}
\end{eqnarray}
Here $a_1$ and $a_2$ are uniquely determined by $M$.
Applying the Cayley-Hamilton Theorem to $M$, we have
$$
M^4-a_1M^3+a_2M^2-a_1M+I=0,
$$
which is equivalent to
$$
(M+I)[M^3-(a_1+1)M^2+(a_1+a_2+1)M-(2a_1+a_2+1)I]=-(2a_1+a_2+2)I.
$$
Then $M+I$ is invertible if and only if $2a_1+a_2+2\ne0$.
In such a case, we have
$$
(M+I)^{-1}=-\frac{1}{2a_1+a_2+2}[M^3-(a_1+1)M^2+(a_1+a_2+1)M-(2a_1+a_2+1)I].
$$
Hence by (\ref{X.eq1}), we have
\begin{equation}\label{X.eq2}
(\tr X)X= M+I+\sigma_2A(M+I)^{-1}
\end{equation}

Since $\sigma_2=(x_1+\frac{1}{x_1})(x_2+\frac{1}{x_2})$, we have
\begin{eqnarray}
    \sigma_2^2&=&(x_1+{1\over x_1})^2(x_2+{1\over x_2})^2
    \nonumber\\
    &=&(x_1^2+{1\over x_1^2}+2)(x_2^2+{1\over x_2^2}+2)
    \nonumber\\
    &=&(x_1^2+{1\over x_1^2})(x_2^2+{1\over x_2^2})
    +2(x_1^2+{1\over x_1^2}+x_2^2+{1\over x_2^2})+4
    \nonumber\\
    &=&2a_1+a_2+2,\label{sigma2}
\end{eqnarray}
where we have used (\ref{a1.dim4}) and (\ref{a2.dim4}) in the last equality.
Moreover, we have
\begin{eqnarray}
    (\tr X)^2&=&(x_1+{1\over x_1}+x_2+{1\over x_2})^2
    \nonumber\\
    &=&x_1^2+{1\over x_1^2}+x_2^2+{1\over x_2^2}
    +2(x_1+{1\over x_1})(x_2+{1\over x_2})+4
    \nonumber\\
    &=&a_1+2\sigma_2+4
    \nonumber\\
    &=&a_1+4\pm2\sqrt{2a_1+a_2+2}.\label{trX.4}
\end{eqnarray}
Now we can characterize the $4\times4$ symplectic matrix
whether it has a real symplectic square root.
First, for the general cases, we have the following characterization:
\begin{lemma}\label{obstruction.4d.2}
Suppose $2a_1+a_2+2\ne0$ and
$a_1+4+2\sqrt{2a_1+a_2+2}\ne0$, we have:

    (i) If $2a_1+a_2+2<0$, $M$ has no real symplectic square root.
     
    (ii) If $2a_1+a_2+2>0$ and $a_1+4+2\sqrt{2a_1+a_2+2}<0$,
    $M$ has no real symplectic square root.

    (iii) If $2a_1+a_2+2>0$ and $a_1+4+2\sqrt{2a_1+a_2+2}>0$,
    there exists at least one real symplectic square root of $M$.
\end{lemma}
\begin{proof}
    (i) follows from $\sigma_2\in\mathbb{R}$ and (\ref{sigma2}),
    and (ii) follows from $\tr X\in\mathbb{R}$ and (\ref{trX.4}).

    (iii) By the argument above, $\sigma_2,\tr X\in\mathbb{R}$ in this case.
    By (\ref{X.eq2}), $X$ is a real matrix.
    Moreover, it's the real symplectic square root of $M$.
\end{proof}





\begin{remark}
The condition $a_1+2a_2+2<0$ in (i) implies $M$ possesses one pair of negative eigenvalues, and the another pair its eigenvalues are positive or on $\mathbb{U}\backslash\{\pm1\}$. Thus $M\in\mathcal{EH}^-$ or $M\in\mathcal{H}^{-+}$.

$a_1+4+2\sqrt{2a_1+a_2+2}<0$ in (ii) implies
$a_2<{1\over4}a_1^2+2$.
However, if $M$ possesses a quadruple of eigenvalues
away from $\mathbb{U}\cup\mathbb{R}$, there must
hold $a_2>{1\over4}a_1^2+2$.
Thus under the conditions in (ii), $M$ possesses two pair of eigenvalues on $\mathbb{U}\cup\mathbb{R}$.
Furthermore, $a_1<-4$ implies $M$ possesses one pair
of negative eigenvalues.

Now we can suppose $\sigma(M)=\{\mu_1,{1\over\mu_1},\mu_2,{1\over\mu_2}\}$ where $\mu_1\in\mathbb{R}^-$ and $\mu_2\in\mathbb{U}\cup\mathbb{R}$.
Then $2a_1+a_2+2=(\mu_1+{1\over\mu_1}+2)(\mu_2+{1\over\mu_2}+2)>0$
and $\mu_1+{1\over\mu_1}+2<0$
implies $\mu_2+{1\over\mu_2}+2<0$.
Hence $\mu_2\in\mathbb{R}^-$.Therefore,
$M\in\mathcal{H}^{--}$.

Accordingly, if $M$ has a real symplectic square root,
then $M$ cannot belong to $\mathcal{H}^{--}, \mathcal{EH}^{-}$ and $\mathcal{H}^{-+}$.
\end{remark}

\medskip

Now we consider the degenerate cases:

(1) $2a_1+a_2+2=0$;

(2) $a_1+4+2\sqrt{2a_1+a_2+2}=0$.

\noindent The first case corresponds to $M$ having repeated eigenvalues of $-1$, and the second case corresponds to $M$ having repeated pairs of negative eigenvalues.

For Case (1),
 we have
 \begin{lemma}\label{eig.-1.symplectic.sqrt}
     Suppose $2a_1+a_2+2=0$. Then $M$ possesses eigenvalue $-1$.
     Furthermore, we have

     (i) If $a_1+4<0$, then $-1$ has a multiplicity of $2$. Then $M$ has no real symplectic square root.

     (ii) If $a_1+4>0$, then $-1$ has a multiplicity of $2$.
     Then $M$ has a real symplectic square root if and only if $\dim\ker(M+I)=2$.

     (iii) If $a_1+4=0$, then $-1$ has a multiplicity of $4$.
     Then $M$ has a real symplectic square root if and only if 
     $M\sim N_1(-1,b)\diamond N_1(-1,b)$
     for some $b\in\mathbb{R}$.
\end{lemma}
\begin{proof}
    Suppose $\sigma(M)=\{\mu_1,{1\over\mu_1},\mu_2,{1\over\mu_2}\}$ for some $\mu_1,\mu_2\in\mathbb{C}$.
Then we have
\begin{eqnarray}
    0&=&2a_1+a_2+2
    \nonumber\\
    &=&2(\mu_1+{1\over\mu_1}+\mu_2+{1\over\mu_2})+
    [(\mu_1+{1\over\mu_1})(\mu_2+{1\over\mu_2})+2]+2
    \nonumber\\
    &=&(\mu_1+{1\over\mu_1}+2)(\mu_2+{1\over\mu_2}+2),
\end{eqnarray}
which implies $\mu_1=-1$ or $\mu_2=-1$.
Without loss of generality, we suppose $\mu_1=-1$.
Hence, $\mu_2\in\mathbb{U}\cup\mathbb{R}$.

If $a_1+4\ne0$, we have $\mu_2\ne-1$. 
Thus $-1$ is an eigenvalue of $M$ with multiplicity of $2$.
According to Theorem 3 in p.36 of \cite{Lon},
there exists $P\in\Sp(4)$ such that
$$
PMP^{-1}=N_1(-1,b)\diamond B,
$$
where $B=D(\mu_2)$ if $\mu_2\in\mathbb{R}$;
and $B=R(\theta)$ if $\mu_2\in\mathbb{U}$.

Furthermore, if $a_1+4<0$, we have $\mu_2<0$.
Hence, $M$ has no real symplectic square root.

If $a_1+4>0$, we have $\mu_2>0$ or $\mu_2\in\mathbb{U}\backslash\{\pm1\}$.
Hence $B$ has at least one real symplectic square root.
Then $A$ has at least one real symplectic square root
if and only if $b=0$, i.e., $\dim\ker(M+I)=2$.

If $a_1+4=0$, we have $\mu_2=-1$. 
Thus $-1$ is an eigenvalue of $M$ with multiplicity of $4$.
Now suppose $X$ is a real symplectic square root of
$M$.
Since each eigenvalue of $X$ is a square root of
some eigenvalue of $M$, we have
$\sigma(X)=\{\pm\sqrt{-1},\pm\sqrt{-1}\}$.
Then by Theorem 11 in p.34 of \cite{Lon},
there exists $P\in\Sp(4,\mathbb{R})$ such that
$$
PXP^{-1}= N_2(\omega,b):=\left(\begin{matrix}R(\theta) & B\cr
                                           0      & R(\theta)\cr\end{matrix}\right),
$$
for $\omega=\sqrt{-1},\theta={\pi\over2}$ or $\omega=-\sqrt{-1},\theta={3\pi\over2}$.
Here $B=\begin{pmatrix}b_1&b_2\cr b_3&b_4\end{pmatrix}$ satisfying $b_1+b_4=0$.

Therefore, we have
\begin{eqnarray}\label{-1.multiplicity.4}
PMP^{-1}&=&(PXP^{-1})^2
\nonumber\\
&=&\left(\begin{matrix}R(\theta) & B\cr
                                           0      & R(\theta)\cr\end{matrix}\right)^2
\nonumber\\
&=&\left(\begin{matrix}R(\theta)^2 & R(\theta)B+BR(\theta)\cr
                                           0      & R(\theta)^2\cr\end{matrix}\right)       
\nonumber\\
&=&\left(\begin{matrix}-I_2 & \pm(b_2-b_3)I_2\cr
                                           0      & -I_2\cr\end{matrix}\right)
\nonumber\\
&=&N_1(-1,\pm(b_2-b_3))\diamond N_1(-1,\pm(b_2-b_3)),
\end{eqnarray}
where we have used $\omega=\pm\sqrt{-1}$ and $b_1+b_4=0$ in the second last equality.
Hence, if $b_2-b_3\ne0$, $\dim\ker(M+I)=2$,
and if $b_2-b_3=0$, $\dim\ker(M+I)=4$.

Conversely, if $M\sim N_1(-1,b)\diamond N_1(-1,b)$,
we can construct the real symplectic square root
of $M$ by using (\ref{-1.multiplicity.4}).


\end{proof}

\begin{remark}\label{rm2.4}
    Since $\dim_{\mathbb{C}}\ker(M+I)=\dim_{\mathbb{C}}\ker(X+\sqrt{-1}I)+\dim_{\mathbb{C}}\ker(X-\sqrt{-1}I)\ge2$, 
    a necessary condition for $M$ to have a square root is $\dim\ker(M+I)\ge2$.
    However, this is not a sufficient condition.
    For example, $N_1(-1,1)\diamond N_1(-1,-1)$ has no real symplectic square root.
\end{remark}

\medskip
For Case (2), first note that $2a_1+a_2+2<0$
is impossible. 
In fact, $a_1+4+2\sqrt{2a_1+a_2+2}=0$ implies $a_1<0$
    and $a_2={1\over4}a_1^2+2$.
    Hence $2a_1+a_2+2={1\over4}(a_1+4)^2\ge0$.
In this case, we have
 \begin{lemma}
     Suppose $2a_1+a_2+2>0$ and $a_1+4+2\sqrt{2a_1+a_2+2}=0$. 
     Then $-1\notin\sigma(M)$ and $M$ possesses a pair of negative eigenvalues with multiplicity $2$.
Then $M$ has at least one real symplectic square root if and only if
$$
M\sim D(\lambda)\diamond D(\lambda),
$$
for some $\lambda<0$.
 \end{lemma}
 \begin{proof}
     Suppose $\sigma(M)=\{\lambda_1,{1\over\lambda_1},\lambda_2,{1\over\lambda_2}\}$ for some $\lambda_1,\lambda_2\in\mathbb{C}$.
Then $2a_1+a_2+2=(\lambda_1+{1\over\lambda_1}+2)(\lambda_2+{1\over\lambda_2}+2)>0$
Then we have
\begin{eqnarray}
    0&=&(a_1+4)^2-4(2a_1+a_2+2)
    \nonumber\\
    &=&\Big[(\lambda_1+{1\over\lambda_1}+2)+(\lambda_2+{1\over\lambda_2}+2)\Big]^2-4(\lambda_1+{1\over\lambda_1}+2)(\lambda_2+{1\over\lambda_2}+2)
\nonumber\\
&=&\Big[(\lambda_1+{1\over\lambda_1}+2)-(\lambda_2+{1\over\lambda_2}+2)\Big]^2
\nonumber\\
&=&(\lambda_1-\lambda_2)^2(1-{1\over\lambda_1\lambda_2})^2
\end{eqnarray}
which implies $\lambda_1=\lambda_2$ or $\lambda_1={1\over\lambda_2}$.
Without loss of generality, we suppose $\lambda_1=\lambda_2=\lambda$.
Furthermore, $a_1+4<0$ implies $\Re\lambda<-1$, 
and hence $\lambda<-1$.

Now suppose $X$ is a real symplectic square root of
$M$.
Since each eigenvalue of $X$ is a square root of
some eigenvalue of $M$, we have
$\sigma(X)=\{\pm\mu\sqrt{-1},\pm{1\over\mu}\sqrt{-1}\}$, where $\mu=\sqrt{|\lambda|}$.
Then by Theorem 2 on p.36 of \cite{Lon},
there exists $P\in\Sp(4,\mathbb{R})$ such that
$$
PXP^{-1}= \begin{pmatrix}
    \mu J_2 & 0\cr 0& {1\over\mu}J_2
\end{pmatrix}.
$$
Therefore, 
we have
\begin{eqnarray}\label{one.pair.multiplicity.2}
PMP^{-1}&=&(PXP^{-1})^2
\nonumber\\
&=&\left(\begin{matrix}\mu J_2 & 0\cr
                                           0      & {1\over\mu}J_2\cr\end{matrix}\right)^2
\nonumber\\
&=&\left(\begin{matrix}-\mu^2 I_2 & 0\cr
                                           0      & -{1\over\mu^2}I_2\cr\end{matrix}\right)       
\nonumber\\
&=&\left(\begin{matrix}\lambda I_2 & 0\cr
                                           0      & {1\over\lambda}I_2\cr\end{matrix}\right)
\nonumber\\
&=&D(\lambda)\diamond D(\lambda).
\end{eqnarray}
where we have used $\lambda=-\mu^2$ in the second last equality.

Conversely, matrix $D(\lambda)\diamond D(\lambda)$
has real symplectic square root
$\begin{pmatrix}
    \mu J_2 & 0\cr 0& {1\over\mu}J_2
\end{pmatrix}$.
\end{proof}

The preceding lemmas cover both the generic situation and the two degenerate cases in which the formula for the square root ceases to be directly applicable. We summarize these conclusions in the following criterion for $4\times 4$ symplectic matrices.

\begin{theorem}\label{Thm:Sp4.square.root} 
Let \(M\in {\rm Sp}(4)\), and suppose that the characteristic polynomial of 
\(M\) is 
\[ p_M(t)=t^4-a_1t^3+a_2t^2-a_1t+1 . 
\] 
Set 
\begin{eqnarray*} 
\alpha&=&a_1+4,\\
\beta&=&2a_1+a_2+2 .
\end{eqnarray*}
When \(\beta\geq 0\), we also set \[ \Theta=\alpha+2\sqrt{\beta}, \] where \(\sqrt{\beta}\) denotes the nonnegative square root. 
A necessary condition for $M$ to have a real symplectic square root is
$\beta\ge0$.

Moreover, \(M\) has a real symplectic square root if and only if one of the following mutually exclusive cases holds: 
\begin{enumerate} 
\item[(i)] \(\beta>0\) and \(\Theta>0\); 
\item[(ii)] \(\beta>0\), \(\Theta=0\), and \(M\sim D(\lambda)\diamond D(\lambda) \) for some \(\lambda<0\);
\item[(iii)] \(\beta=0\), \(\alpha>0\), and $\dim\ker(M+I)=2$;
\item[(iv)] \(\beta=0\), \(\alpha=0\), and $M\sim N_1(-1,b)\diamond N_1(-1,b)$ for some \(b\in\mathbb R\). 

\end{enumerate} 
In all remaining cases, \(M\) has no real symplectic square root. 
\end{theorem}

\subsection{Real symplectic square roots of $6\times6$ symplectic matrices}

We next consider the square-root problem in dimension six. 
As in the \(4\times4\) case, 
the Cayley--Hamilton theorem gives an algebraic relation 
between a symplectic square root $X$ and the given matrix $M$. 
In dimension six, however, three spectral parameters are involved, 
and the resulting conditions are naturally encoded by a cubic polynomial.

Suppose \(M,X\in{\rm Sp}(6)\) and \[ X^2=M. \] 
Let 
\[ \sigma(X)=\left\{x_1,x_1^{-1}, x_2,x_2^{-1}, x_3,x_3^{-1}\right\}, 
\] 
and set 
\[ y_i=x_i+x_i^{-1},\qquad i=1,2,3. 
\] 
Define 
\begin{equation} \label{def.sigma}
\sigma_1=y_1+y_2+y_3,\qquad \sigma_2=y_1y_2+y_2y_3+y_3y_1,\qquad \sigma_3=y_1y_2y_3 . 
\end{equation} 
Then the characteristic polynomial of \(X\) can be written as 
\[ 
p_X(t) = \prod_{i=1}^3\left(t^2-y_it+1\right) = t^6-\sigma_1t^5+(\sigma_2+3)t^4-(\sigma_3+2\sigma_1)t^3 +(\sigma_2+3)t^2-\sigma_1t+1 . 
\] 
By the Cayley--Hamilton theorem, 
\[ 
X^6-\sigma_1X^5+(\sigma_2+3)X^4-(\sigma_3+2\sigma_1)X^3 +(\sigma_2+3)X^2-\sigma_1X+I=0. 
\] 
Using \(X^2=M\), we obtain 
\[ 
M^3-\sigma_1XM^2+(\sigma_2+3)M^2 -(\sigma_3+2\sigma_1)XM +(\sigma_2+3)M-\sigma_1X+I=0. 
\] 
Equivalently, 
\begin{equation} \label{eq1.of.X}
X\left[\sigma_1(M+I)^2+\sigma_3M\right] = (M+I)^3+\sigma_2M(M+I). 
\end{equation} 
Left-multiplying both side of \eqref{eq1.of.X} and substituting $X^2=M$ into it, we obtain
\begin{equation} \label{eq2.of.X}
X\left[(M+I)^3+\sigma_2M(M+I)\right] = M[\sigma_1(M+I)^2+\sigma_3M]. 
\end{equation} 
Taking $\sigma_1$ times \eqref{eq2.of.X} minus \eqref{eq1.of.X} times $M+I$,
we get
\begin{equation} \label{eq.of.X}
X[(\sigma_1\sigma_2-\sigma_3)M(M+I)] = \sigma_1\sigma_3M^2+(\sigma_1^2-\sigma_2)M(M+I)^2-(M+I)^4. 
\end{equation}

We now express these parameters in terms of the coefficients of the characteristic polynomial of \(M\). Write 
\begin{equation} \label{pM}
p_M(t)=t^6-a_1t^5+a_2t^4-a_3t^3+a_2t^2-a_1t+1 . 
\end{equation} 
Since \(M=X^2\), the eigenvalues of \(M\) are \[ x_1^2,x_1^{-2}, x_2^2,x_2^{-2}, x_3^2,x_3^{-2}. \] 
Set 
\[ \mu_i=x_i^2+x_i^{-2}=y_i^2-2,\qquad i=1,2,3. \] 
Then 
\[ p_M(t)=\prod_{i=1}^3(t^2-\mu_it+1). 
\] 
Hence 
\begin{eqnarray} 
a_1&=&\mu_1+\mu_2+\mu_3, \label{a1.dim6}\\
a_2&=&\mu_1\mu_2+\mu_2\mu_3+\mu_3\mu_1+3,\label{a2.dim6}\\
a_3&=&\mu_1\mu_2\mu_3+2(\mu_1+\mu_2+\mu_3). \label{a3.dim6}
\end{eqnarray}
Substituting \(\mu_i=y_i^2-2\) and expressing the result in terms of \(\sigma_1,\sigma_2,\sigma_3\), we obtain 
\begin{eqnarray*} 
a_1&=&\sigma_1^2-2\sigma_2-6,\\
a_2&=&\sigma_2^2-2\sigma_1\sigma_3-4\sigma_1^2+8\sigma_2+15,
\\
a_3&=&\sigma_3^2-2\sigma_2^2+4\sigma_1\sigma_3 +6\sigma_1^2-12\sigma_2-20.  
\end{eqnarray*}

On the other hand,
applying the Cayley-Hamilton Theorem to $M$, we have
$$
M^6-a_1M^5+a_2M^4-a_3M^3+a_2M^2-a_1M+I=0,
$$
which is equivalent to
\begin{eqnarray*}
&&(M+I)\Big[M^5-(a_1+1)M^4+(a_1+a_2+1)M^3-(a_1+a_2+a_3+1)M^2
\\
&&\quad\quad+(a_1+2a_2+a_3+1)M-(2a_1+2a_2+a_3+1)I\Big]
\\
&&=-(2a_1+2a_2+a_3+2)I.
\end{eqnarray*}
Then $M+I$ is invertible if and only if $2a_1+2a_2+a_3+2\ne0$.
In such a case, we have
\begin{eqnarray}
(M+I)^{-1}&=&-\frac{1}{2a_1+2a_2+a_3+2}\Big[M^5-(a_1+1)M^4+(a_1+a_2+1)M^3-(a_1+a_2+a_3+1)M^2
\nonumber\\
&&\quad\quad+(a_1+2a_2+a_3+1)M-(2a_1+2a_2+a_3+1)I\Big].
\label{M+I.inverse}
\end{eqnarray}
Therefore, if $2a_1+2a_2+a_3+2\ne0$ and $\sigma_1\sigma_2-\sigma_3\ne0$,
we can solve for $X$ form \eqref{eq.of.X}, which is given by
\begin{equation}\label{X}
X=\frac{1}{\sigma_1\sigma_2-\sigma_3}\Big[\sigma_1\sigma_3M^2+(\sigma_1^2-\sigma_2)M(M+I)^2-(M+I)^4\Big]M^{-1}(M+I)^{-1}.
\end{equation}

Next, we need to relate $\sigma_1\sigma_2-\sigma_3$ to some formula in terms of the data of $M$, such as $a_1,a_2,a_3$.
By \eqref{def.sigma}, we have
$$
\sigma_1\sigma_2-\sigma_3=(y_1+y_2+y_3)(y_1y_2+y_2y_3+y_3y_1)-y_1y_2y_3
=(y_1+y_2)(y_2+y_3)(y_3+y_1).
$$
Note that
\begin{equation}\label{factor.of.Delta.q}
\Delta_q=(y_1^2-y_2^2)^2(y_2^2-y_3^2)^2(y_3^2-y_1^2)^2
=(\sigma_1\sigma_2-\sigma_3)^2(y_1-y_2)^2(y_2-y_3)^2(y_3-y_1)^2
\end{equation}
is the discriminant of the cubic polynomial
\begin{equation}\label{cubic.polynomial}
q(t)=\prod_{i=1}^3(t-y_i^2)=t^3-(a_1+6)t^2+(4a_1+a_2+9)t-(2a_1+2a_2+a_3+2),
\end{equation}
where we have used $\mu_i=y_i^2-2,i=1,2,3$ and \eqref{a1.dim6}-\eqref{a3.dim6}.
We define
\begin{eqnarray}
    \alpha &=& a_1+6,\label{alpha}\\
    \beta &=& 4a_1+a_2+9,
    \\
    \gamma &=& 2a_1+2a_2+a_3+2.\label{gamma}
\end{eqnarray}
Then the discriminant $\Delta_q$ of $q(t)$ in \eqref{cubic.polynomial}
can be represented in terms of $\alpha,\beta,\gamma$, as shown below:
\begin{equation}\label{Delta.q}
    \Delta_q=\alpha^2\beta^2-4\beta^3-4\alpha^3\gamma
-27\gamma^2+18\alpha\beta\gamma.
\end{equation}
Now we can characterize the $6\times6$ symplectic matrix
whether it has a real symplectic square root under some non-degenerate conditions.

\begin{lemma}\label{Lm:2.9}
Using the notation introduced above,
assume that
\[
\gamma\neq0
\qquad\text{and}\qquad
\Delta_q\neq0.
\]
Then the following statements hold.

\begin{enumerate}
\item[\rm (i)]
If $\gamma<0$, then $M$ has no real symplectic square root.

\item[\rm (ii)]
If $\gamma>0$ and $\Delta_q<0$,
then $M$ has a real symplectic square root.

\item[\rm (iii)]
Suppose that $\gamma>0$ and $\Delta_q>0$.
Then $M$ has a real symplectic square root if and only if
\[
\alpha>0
\qquad\text{and}\qquad
\beta>0,
\]
or equivalently,
\[
a_1+6>0
\qquad\text{and}\qquad
a_2+4a_1+9>0.
\]
\end{enumerate}
\end{lemma}

\begin{proof} First, we have
\begin{eqnarray*}
    \sigma_3^2&=&y_1^2y_2^2y_3^2\\
    &=&(\mu_1+2)(\mu_2+2)(\mu_3+2)
    \nonumber\\
    &=&\mu_1\mu_2\mu_3+2(\mu_1\mu_2+\mu_2\mu_3+\mu_3\mu_1)+4(\mu_1+\mu_2+\mu_3)+8
    \nonumber\\
    &=&(a_3-2a_1)+2(a_2-3)+4a_1+8
    \nonumber\\
    &=&2a_1+2a_2+a_3+2,\\
    &=&\gamma,
\end{eqnarray*}
where we have used $y_i^2=\mu_i+2$ in the second equality, and used \eqref{a1.dim6}-\eqref{a3.dim6} in the fourth equality.
The existence of a square root of $M$ implies that $\sigma_3\in\mathbb{R}$.
Thus $\gamma\geq0$ is necessary for the existence of a real
symplectic square root. This proves {\rm (i)}.

We then establish a construction that will be used in the proofs
of {\rm (ii)} and {\rm (iii)}. 
By \eqref{Delta.q}, $\Delta_q\ne0$ implies $\sigma_1\sigma_2-\sigma_3\neq0$.
Then we can define a rational function with respect to $t$:
\begin{equation}\label{F}
F(t)=\frac{
\sigma_1\sigma_3t^2
+(\sigma_1^2-\sigma_2)t(t+1)^2
-(t+1)^4
}{
(\sigma_1\sigma_2-\sigma_3)t(t+1)
}.
\end{equation}
Since
\[
\gamma=p_M(-1)=\det(M+I)\neq0,
\]
the matrix $M+I$ is invertible. Since $M$ is symplectic, $M$
is also invertible. Thus matrix $F(M)$ is well defined. 
Set
\[
X=F(M).
\]

Recall that $\gamma=\sigma_3^2$. Moreover,
from \eqref{a1.dim6}-\eqref{a3.dim6} and \eqref{alpha}-\eqref{gamma}, we have
\[
\alpha=\sigma_1^2-2\sigma_2,\qquad
\beta=\sigma_2^2-2\sigma_1\sigma_3.
\]
Then a direct calculation gives
\begin{equation}\label{check.for.sqare}
\Bigl[
\sigma_1\sigma_3t^2
+(\sigma_1^2-\sigma_2)t(t+1)^2
-(t+1)^4
\Bigr]^2\
-t\Bigl[
(\sigma_1\sigma_2-\sigma_3)t(t+1)
\Bigr]^2\
=
\bigl[(t+1)^2-\sigma_1^2t\bigr]p_M(t).
\end{equation}
Applying \eqref{check.for.sqare} to $M$ and using the Cayley-Hamilton theorem,
we obtain
\begin{equation}\label{FM.squre}
F(M)^2=M.
\end{equation}

Then we verify that $F(M)$ is symplectic. It follows
directly from \eqref{F} that
\begin{equation}\label{F.t.inv}
F(t^{-1})=t^{-1}F(t).
\end{equation}
For any real rational function $R$ defined on $\sigma(M)$, the
symplectic identity
\[
M^TJ=JM^{-1}
\]
implies
\[
R(M)^TJ=JR(M^{-1}).
\]
Therefore, by \eqref{FM.squre} and \eqref{F.t.inv},
\[
\begin{aligned}
F(M)^TJF(M)
&=JF(M^{-1})F(M)\
&=JM^{-1}F(M)^2\
&=J.
\end{aligned}
\]
Thus $F(M)\in\operatorname{Sp}(6)$.
It remains to check under what conditions,
$F(M)$ is a real matrix.

We now prove {\rm (ii)}. 
If $\Delta_q<0$, then \eqref{cubic.polynomial} has one
real root $r$ and a pair of nonreal conjugate roots
$(z,\overline z)$. Since their product is $\gamma>0$,
\[
r|z|^2=\gamma>0,
\]
and hence $r>0$. Choose
\[
y_1=\sqrt r\in\mathbb R.
\]
Choose a square root $y_2$ of $z$, and set
\[
y_3=\overline{y_2}.
\]
Then
\[
y_2^2=z,\qquad
y_3^2=\overline z,
\]
and the corresponding quantities
$\sigma_1,\sigma_2,\sigma_3$ are real. The construction
\eqref{F}-\eqref{F.t.inv} therefore yields a real symplectic square root
of $M$. This proves {\rm (ii)}.

Finally, suppose that $\gamma>0$ and $\Delta_q>0$. Then \eqref{cubic.polynomial}
has three distinct real roots. Since their product is positive,
their signs are either
\[
(+,+,+)
\qquad\text{or}\qquad
(+,-,-).
\]

If all three roots are positive, then
\[
\alpha>0
\qquad\text{and}\qquad
\beta>0.
\]
Conversely, suppose that the roots of \eqref{cubic.polynomial} are
\[
r,\quad -u,\quad -v,
\qquad r,u,v>0.
\]
Then
\[
\alpha=r-u-v,
\qquad
\beta=uv-r(u+v).
\]
If $\alpha>0$, then $r>u+v$, and consequently
\[
\beta
<uv-(u+v)^2
<0.
\]
Thus the sign pattern (+,-,-) is incompatible with $\alpha>0$ and $\beta>0$.
It follows that the three roots of \eqref{cubic.polynomial} are all positive if and
only if 
\[
\alpha>0
\qquad\text{and}\qquad
\beta>0.
\]
When these two inequalities hold, choose
$y_1,y_2,y_3\in\mathbb R$ whose squares are the three positive
roots of \eqref{cubic.polynomial}. Then $\sigma_1,\sigma_2,\sigma_3$ are real, and
the construction above gives a real symplectic square root of $M$.

Conversely, suppose that the roots of \eqref{cubic.polynomial} have signs
(+,-,-). If $M$ had a real symplectic square root, then
\[
t^3-\sigma_1t^2+\sigma_2t-\sigma_3
\]
would be a real polynomial with roots $y_1,y_2,y_3$. The positive
root of \eqref{cubic.polynomial} has real square roots, whereas the two distinct
negative roots have purely imaginary square roots with different
absolute values. No choice of these two imaginary square roots can
form a complex conjugate pair. Hence
\[
{y_1,y_2,y_3}
\]
cannot be invariant under complex conjugation, contradicting the
fact that the polynomial above has real coefficients. Therefore
$M$ has no real symplectic square root in this case.

Since
\[
\alpha=a_1+6,
\qquad
\beta=4a_1+a_2+9,
\]
the assertion in {\rm (iii)} follows.
\end{proof}

\begin{remark}
The condition $\gamma=2a_1+2a_2+a_3+2<0$ in (i) implies that $M$ possesses at least one pair of negative eigenvalues with multiplicity $1$. 
Conditions $\gamma>0$ and $\Delta_q<0$ in (ii) imply that $M$ possesses a pair of positive eigenvalues
and a quadruple of eigenvalues
away from $\mathbb{U}\cup\mathbb{R}$.
Conditions $\gamma>0$ and $\Delta_q>0$ in (iii) imply that $M$ possesses three distinct pairs of real eigenvalues, and at least one pair of which are positive.
If additionally $\alpha>0$ and $\beta>0$, then all three pairs are positive.
\end{remark}

Next we consider the two degenerate cases:

(1) $2a_1+2a_2++a_3+2=0$;

(2) $\Delta_q=0$.

\noindent The first case corresponds to $M$ having repeated eigenvalues of $-1$, and the second case corresponds to $M$ having repeated pairs of eigenvalues.
Prior to treating the two cases, we state the following lemma.
\begin{lemma}[Square roots under spectral $\diamond$-decompositions]
\label{lem:sqrt.under.decompositions}
Let $M\in\operatorname{Sp}(2n)$. Suppose that, for some
$M_i\in\operatorname{Sp}(2n_i), i=1,2$, with
$n_1+n_2=n$, one has
\[
M\sim M_1\diamond M_2,
\]
and
\[
\sigma(M_1)\cap\sigma(M_2)=\emptyset.
\]
Then $M$ has a real symplectic square root if and only if both
$M_1$ and $M_2$ have real symplectic square roots.
\end{lemma}

\begin{proof}
The existence of a real symplectic square root is invariant under
symplectic similarity. Indeed, if $S\in\operatorname{Sp}(2n)$
and $X^2=M$, then $(SXS^{-1})^2=SMS^{-1}$,
and $SXS^{-1}$ is again symplectic. 
Hence it suffices to consider
the case
\[
M=M_1\diamond M_2.
\]

First, if both $M_1$ and $M_2$ have real symplectic square roots,
denoted by $X_1$ and $X_2$, respectively,
then we have
\[
(X_1\diamond X_2)^2=X_1^2\diamond X_2^2=
M_1\diamond M_2=M.
\]
Thus $M$ has a real symplectic square root.

Conversely, suppose $X$ is a real symplectic square root of $M$.
The $\diamond$-decomposition of $M$ determines an
$M$-invariant symplectic direct-sum decomposition
\[
\mathbb R^{2n}=V_1\oplus^{\omega}V_2,
\qquad
\dim V_i=2n_i,
\]
such that, after identifying $V_i$ symplectically with
$\mathbb R^{2n_i}$, the restriction of $M$ to $V_i$ is
represented by $M_i$. 
With respect to this decomposition, write
associated with $M_1\diamond M_2$, and write $X$, with respect
to this decomposition, as
\[
X=
\begin{pmatrix}
X_{11}&X_{12}\\
X_{21}&X_{22}
\end{pmatrix}.
\]
Equation $MX=X^3=XM$ gives
\begin{equation}
M_1X_{12}=X_{12}M_2,
\qquad
M_2X_{21}=X_{21}M_1.
\label{Sylv.eq}
\end{equation}
Since $\sigma(M_1)\cap\sigma(M_2)=\emptyset$,
the Sylvester equations in \eqref{Sylv.eq} admit only the zero solution.
Consequently,
\[
X_{12}=0,
\qquad
X_{21}=0,
\]
and hence
\[
X=X_1\diamond X_2
\]
for some real matrices $X_i\in\mathbb{R}^{2n_i\times2n_i}$.

It follows from $X^2=M_1\diamond M_2$ that
\begin{equation}\label{X1.X2.sqare}
X_1^2=M_1,
\qquad
X_2^2=M_2.
\end{equation}
Moreover, the symplectic form associated with the above direct-sum
decomposition is
\[
J_{2n}=J_{2n_1}\oplus J_{2n_2}.
\]
Since $X$ is symplectic, $X^TJ_{2n}X=J_{2n}$.
Using the block diagonal form of $X$, this identity becomes
\[
X_1^TJ_{2n_1}X_1=J_{2n_1},
\qquad
X_2^TJ_{2n_2}X_2=J_{2n_2}.
\]
Thus $X_i\in\operatorname{Sp}(2n_i),\;i=1,2$.
Together with \eqref{X1.X2.sqare}, this shows that $X_i$ is a real symplectic
square root of $M_i$ for $i=1,2$.
\end{proof}

For Degenerate Case (1),
 we have
 \begin{lemma}\label{eig.-1.symplectic.sqrt.dim6}
     Suppose $\gamma=2a_1+2a_2+a_3+2=0$. Then $M$ possesses eigenvalue $-1$.
     Furthermore, we have

     (i) If $\beta=4a_1+a_2+9\ne0$, then eigenvalue $-1$ has a multiplicity of $2$. Then $M$ has a real symplectic square root if and only if 
     $M\sim-I_2\diamond M_4$ for some $M_4\in\Sp(4)$ that also has a real symplectic square root.
    
     (ii) If $\beta=4a_1+a_2+9=0$ and $\alpha=a_1+6<0$, then eigenvalue $-1$ has a multiplicity of $4$.
     Then $M$ has no real symplectic square root.

     (iii) If $\beta=4a_1+a_2+9=0$ and $\alpha=a_1+6>0$, then eigenvalue $-1$ has a multiplicity of $4$.
     Then $M$ has a real symplectic square root if and only if 
     $M\sim N_1(-1,b)\diamond N_1(-1,b)\diamond M_2$
     for some $b\in\mathbb{R}$, and $M_2=D(\lambda)$ for some $\lambda>0$ or
     $M_2=R(\theta)$ for some $\theta\in(0,\pi)\cup(\pi,2\pi)$.
     
     (iv) If $\alpha=\beta=0$, i.e., $a_1=-6,a_2=15,a_3=-20$, then eigenvalue $-1$ has a multiplicity of $6$.
     Then $M$ has a real symplectic square root if and only if 
     $M\sim-I_2\diamond N_1(-1,b)\diamond N_1(-1,b)$
     for some $b\in\mathbb{R}$,
     or
     $$
M\sim \begin{pmatrix}
-1&1&0&0&0&s\\
0&-1&0&0&0&0\\
0&0&-1&s&0&0\\
0&0&0&-1&0&0\\
0&0&-2s&1&-1&0\\
0&-2s&0&0&0&-1
\end{pmatrix},
$$
for $s=\pm1$.
\end{lemma}

\begin{remark}
    Cases (ii)-(iv) imply $\Delta_q=0$. In Cases (i), however, both the cases $\Delta_q=0$ and $\Delta_q\ne0$ can occur.
\end{remark}

\begin{proof}
If $\gamma=0$, by \eqref{pM}, we have
\begin{eqnarray*}
    p_M(t)&=&(t+1)^2[t^4-(a_1+2)t^3+(2a_1+a_2+3)t^2-(a_1+2)t+1]\\
    &=&(t+1)^6-(a_1+6)t(t+1)^4+(4a_1+a_2+9)t^2(t+1)^2
    \\
    &=&(t+1)^6-\alpha t(t+1)^4+\beta t^2(t+1)^2.
\end{eqnarray*}
Then $\beta\ne0$ gives multiplicity $2$ for the eigenvalue $-1$;
$\beta=0, \alpha\ne0$ gives multiplicity $4$;
and $\alpha=\beta=0$ gives multiplicity $6$.

    Suppose $\sigma(M)=\{\lambda_1,{1\over\lambda_1},\lambda_2,{1\over\lambda_2},\lambda_3,{1\over\lambda_3}\}$ for some $\lambda_1,\lambda_2,\lambda_3\in\mathbb{C}$.
Without loss of generality, we suppose $\lambda_1=-1$.

If $\beta\ne0$, we have $\lambda_2,\lambda_3\ne-1$.
Then there exists $b\in\mathbb{R}$ and $M_4\in\Sp(4)$ such that
$$
M\sim N_1(-1,b)\diamond M_4.
$$
By a similar argument to that in Remark \ref{rm2.4},
a necessary condition for $M$ to have a square root is $\dim\ker(M+I)\ge2$.
Note that $\sigma(M_4)=\{\lambda_2,{1\over\lambda_2},\lambda_3,{1\over\lambda_3}\}$, which does not contain $-1$.
Thus, by Lemma \ref{lem:sqrt.under.decompositions}, Theorem \ref{thm:sqrt.dim2} and Theorem \ref{Thm:Sp4.square.root},
if $M$ has a real symplectic square root, if and only $b=0$ and
$M_4\in\Sp(4)$ has a real symplectic square root.
This proves (i).

If $\beta=0$ and $\alpha\ne0$, we can suppose $\lambda_1=\lambda_2=-1$ and $\lambda_3\ne-1$.
Hence, $\lambda_3\in\mathbb{U}\cup\mathbb{R}$.
Then there exists $b\in\mathbb{R}$ and $M_4\in\Sp(4)$ such that
$$
M\sim \widetilde{M}_4\diamond M_2,
$$
where $M_2\in\Sp(2)$ does not have eigenvalue $-1$, and $\widetilde{M}_4\in\Sp(4)$ has only eigenvalues $-1$ with multiplicity $4$.
By Lemma \ref{lem:sqrt.under.decompositions}, $M$ has a real symplectic square root if and only if both $\widetilde{M}_4$ and $M_2$ have real symplectic square root.
Then by Lemma \ref{eig.-1.symplectic.sqrt}, there exists $b\in\mathbb{R}$ such that $\widetilde{M}_4\sim N_1(-1,b)^{\diamond2}$.
Moreover, since $-1\notin\sigma(M_2)$, by Theorem \ref{thm:sqrt.dim2},
$M_2$ has a real symplectic square root if and only if $\tr(M_2)>-2$,
which is equivalent $\alpha=a_1+6>0$.
This proves (ii) and (iii).

We now address Case (iv), which is the more complicated of the cases.

If $\alpha=\beta=0$, then $M$ has eigenvalue $-1$ with a multiplicity of $6$.
Suppose $X$ is a real symplectic square root of $M$.
Then we have $\sigma(X)=\{\pm\sqrt{-1}\}$.
Note that $X$ may not itself be a Wonenburger matrix,
However, up to a suitable symplectic conjugation, we may assume that 
$X=\begin{pmatrix}X_{11} & X_{12}\\X_{21} & X_{11}^T\end{pmatrix}$ is a Wonenburger matrix. 
In that case, we now have $M\sim X^2$.
Since $\sigma(M)=\{-1\}$ implies $\sigma(X)=\{\pm\sqrt{-1}\}$, 
it follows that $\sigma(X_{11})=\{0\}$.
Then, up to a suitable symplectic conjugation, there are only three possible subcases:
$$(1)\; X_{11}=0; \quad(2)\; X_{11}={\rm diag}(J_2(0),0); \quad(3)\; X_{11}=J_3(0).$$

For Subcase (1), if $X_{11}=0$, we have $X_{12}X_{21}=X_{11}^2-I_3=-I_3$,
which implies that both $X_{12}$ and $X_{21}$ are invertible.
Since $X_{12}$ is symmetric, there exists an invertible matrix $R$ such
that $RBR^T={\rm diag}(b_1,b_2,b_3)$ where $b_i=\pm1,i=1,2,3$.
Thus, under the $R_*$ action, we have
\begin{equation*}
        R_*\left(\begin{matrix}X_{11} & X_{12}\\X_{21} & X_{11}^T\end{matrix}\right)
        =\begin{pmatrix}
    0& 0& 0& b_1& 0& 0\\
    0& 0& 0& 0& b_2& 0\\
    0& 0& 0& 0& 0& b_3\\
    -b_1& 0& 0& 0& 0& 0\\
    0& -b_2& 0& 0& 0& 0\\
    0& 0& -b_3& 0& 0& 0
\end{pmatrix}.
    \end{equation*}
Direct computation shows
$$
R_*(X^2)=-I_6,
$$
where we have used $b_i^2=1$ for $i=1,2,3$.
Therefore, we have $M=-I_6$.

For Subcase (2), if $X_{11}={\rm diag}(J_2(0),0)$, we also have $X_{12}X_{21}=X_{11}^2-I_3=-I_3$,
which implies that both $X_{12}$ and $X_{21}$ are invertible.
Suppose $R\in {\rm GL}_3(\mathbb{R})$ commutes with $X_{11}$.
A direct computation shows that $R$ has the form
$$
R=\begin{pmatrix}
    r_{11}& r_{12}& r_{13}\\
    0& r_{11}& 0\\
    0& r_{32}& r_{33}
\end{pmatrix},
$$
where $\det(R)=r_{11}^2r_{33}\ne0$.
Denote by $X_{12}=\{b_{ij}\}_{1\le i,j\le3}$.
By $X_{12}=X_{12}^T$ and $X_{11}X_{12}=X_{12}X_{11}^T$, we have
$$
X_{12}=\begin{pmatrix}
    b_{11}& b_{12}& b_{13}\\
    b_{12}& 0& 0\\
    b_{13}& 0& b_{33}
\end{pmatrix},
$$
Moreover, we denote by $RBR^T=\{\widetilde{b}_{ij}\}_{1\le i,j\le3}$.
Then we have
\begin{eqnarray*}
    \widetilde{b}_{13}&=&\widetilde{b}_{31}=b_{12}r_{11}r_{32}+
    b_{13}r_{11}r_{33}+b_{33}r_{13}r_{33},
    \\
    \widetilde{b}_{23}&=&\widetilde{b}_{32}=0.
\end{eqnarray*}
Since $\det(X_{12})=-b_{12}^2b_{33}\ne0$,
we set $r_{13}=-\frac{b_{12}r_{11}r_{32}+
    b_{13}r_{11}r_{33}}{b_{33}r_{33}}$.
Then we have $\widetilde{b}_{13}=\widetilde{b}_{31}=0$.
Therefore, up to a suitable $R_*$-action, we may assume $b_{13}=0$.
Furthermore, setting $r_{13}=r_{32}=0$,
we obtain
$$
RX_{12}R^T=\begin{pmatrix}
    r_{11}^2b_{11}+2r_{11}r_{12}b_{12}& r_{11}^2b_{12}& 0\\
    r_{11}^2b_{12}& 0& 0\\
    0& 0& r_{33}^2b_{33}
\end{pmatrix}.
$$
Since $\det(X_{12})=-b_{12}^2b_{33}\ne0$,
we set $r_{11}=|b_{12}|^{-1/2}, r_{12}=-\frac{r_{11}b_{11}}{2b_{12}}, r_{33}=|b_{33}|^{-1/2}$.
Then we have
$$
RX_{12}R^T=\begin{pmatrix}
    0& s_1& 0\\
    s_1& 0& 0\\
    0& 0& s_2
\end{pmatrix},
$$
where $s_1,s_2=\pm1$.
Thus, under the $R_*$ action, we have
\begin{equation*}
        R_*\left(\begin{matrix}X_{11} & X_{12}\\X_{21} & X_{11}^T\end{matrix}\right)
        =\begin{pmatrix}
    0& 1& 0& 0& s_1& 0\\
    0& 0& 0& s_1& 0& 0\\
    0& 0& 0& 0& 0& s_2\\
    0& -s_1& 0& 0& 0& 0\\
    -s_1& 0& 0& 1& 0& 0\\
    0& 0& -s_2& 0& 0& 0
\end{pmatrix}.
\end{equation*}
Direct computation shows
$$
R_*(X^2)
=\begin{pmatrix}
    -1& 0& 0& 2s_1& 0& 0\\
    0& -1& 0& 0& 0& 0\\
    0& 0& -1& 0& 0& 0\\
    0& 0& 0& -1& 0& 0\\
    0& -2s_1& 0& 0& -1& 0\\
    0& 0& 0& 0& 0& -1
\end{pmatrix}
=\begin{pmatrix}
    -1& 2s_1\\
    0& -1&
\end{pmatrix}\diamond\begin{pmatrix}
    -1& 0\\
    -2s_1& -1&
\end{pmatrix}\diamond(-I_2).
$$
Since $\begin{pmatrix}
    -1& 0\\
    -2s_1& -1&
\end{pmatrix}$ is symplectic conjugate to $\begin{pmatrix}
    -1& 2s_1\\
    0& -1&
\end{pmatrix}$, we obtain
$$
M\sim R_*(X^2)\sim-I_2\diamond N_1(-1,2s_1)^{\diamond2}.
$$

For Subcase (3), if $X_{11}=J_3(0)$, we now have 
$$X_{12}X_{21}=X_{11}^2-I_3=\begin{pmatrix}
    -1& 0& 1\\
    0& -1& 0\\
    0& 0& -1
\end{pmatrix},$$
which implies that both $X_{12}$ and $X_{21}$ are invertible.
Suppose $R\in {\rm GL}_3(\mathbb{R})$ commutes with $X_{11}$.
A direct computation shows that $R$ has the form
$$
R=\begin{pmatrix}
    r_{11}& r_{12}& r_{13}\\
    0& r_{11}& r_{12}\\
    0& 0& r_{11}
\end{pmatrix},
$$
where $\det(R)=r_{11}^3\ne0$.
Denote by $X_{12}=\{b_{ij}\}_{1\le i,j\le3}$.
By $X_{12}=X_{12}^T$ and $X_{11}X_{12}=X_{12}X_{11}^T$, we have
$$
X_{12}=\begin{pmatrix}
    b_{11}& b_{12}& b_{13}\\
    b_{12}& b_{13}& 0\\
    b_{13}& 0& 0
\end{pmatrix},
$$
Moreover, we denote by $RBR^T=\{\widetilde{b}_{ij}\}_{1\le i,j\le3}$.
Then we have
$$
RX_{12}R^T=\begin{pmatrix}
    \widetilde{b}_{11}& \widetilde{b}_{12}& \widetilde{b}_{13}\\
    \widetilde{b}_{12}& \widetilde{b}_{13}& 0\\
    \widetilde{b}_{13}& 0& 0
\end{pmatrix},
$$
where
\begin{eqnarray*}
    \widetilde{b}_{11}&=&r_{11}^2b_{11}+2r_{11}r_{12}b_{12}+(r_{12}^2+2r_{11}r_{13})b_{13},
    \\
    \widetilde{b}_{12}&=&r_{11}^2b_{12}+2r_{11}r_{12}b_{13},
    \\
    \widetilde{b}_{13}&=&r_{11}^2b_{13}.
\end{eqnarray*}
Since $\det(X_{12})=-b_{13}^3\ne0$,
we set $r_{11}=|b_{13}|^{-1/2}, r_{12}=-\frac{r_{11}b_{12}}{2b_{13}}$ and 
$r_{13}=-\frac{r_{11}^2b_{11}+2r_{11}r_{12}b_{12}+r_{12}^2b_{13}}{2r_{11}b_{13}}$.
Then we have
$$
RX_{12}R^T=\begin{pmatrix}
    0& 0& s\\
    0& s& 0\\
    s& 0& 0
\end{pmatrix},
$$
where $s=\pm1$.
Thus, under the $R_*$ action, we have
\begin{equation*}
        R_*\left(\begin{matrix}X_{11} & X_{12}\\X_{21} & X_{11}^T\end{matrix}\right)
        =\begin{pmatrix}
    0& 1& 0& 0& 0& s\\
    0& 0& 1& 0& s& 0\\
    0& 0& 0& s& 0& 0\\
    0& 0& -s& 0& 0& 0\\
    0& -s& 0& 1& 0& 0\\
    -s& 0& s& 0& 1& 0
\end{pmatrix}.
\end{equation*}
Direct computation shows
$$
R_*(X^2)
=\begin{pmatrix}
    -1& 0& 2& 0& 2s& 0\\
    0& -1& 0& 2s& 0& 0\\
    0& 0& -1& 0& 0& 0\\
    0& 0& 0& -1& 0& 0\\
    0& 0& -2s& 0& -1& 0\\
    0& -2s& 0& 2& 0& -1
\end{pmatrix}
$$
Let
$$
Q_1=\begin{pmatrix}
    2^{-\frac{3}{4}}& 0& 0\\
    0& 2^{-\frac{1}{4}}& 0\\
    0& 0& 2^{\frac{1}{4}}
\end{pmatrix},\quad
Q_2=\begin{pmatrix}
    1& 0& 0\\
    0& 0& 1\\
    0& 1& 0
\end{pmatrix},
$$
and $Q=Q_2Q_1$. Then we have
\[
Q_*(R_*(X^2))
=\begin{pmatrix}
-1&1&0&0&0&s\\
0&-1&0&0&0&0\\
0&0&-1&s&0&0\\
0&0&0&-1&0&0\\
0&0&-2s&1&-1&0\\
0&-2s&0&0&0&-1
\end{pmatrix}.
\]

Conversely, each of the two normal forms listed in {\rm (iv)}
does admit a real symplectic square root. Indeed, the first one is
obtained from Subcase {\rm (1)} when $X_{11}=0$, or from
Subcase {\rm (2)} when
$X_{11}=\operatorname{diag}(J_2(0),0)$. The second one is exactly
the square of the Wonenburger normal form obtained in Subcase
{\rm (3)}, after the $Q_*$-action above. Therefore both
normal forms have real symplectic square roots, and the proof of
{\rm (iv)} is complete.
\end{proof}

For Degenerate Case (2), first note that $2a_1+2a_2+a_3+2<0$
is impossible if $M$ admits a real symplectic square root,
and the case $2a_1+2a_2+a_3+2=0$ has already been handled by Lemma \ref{eig.-1.symplectic.sqrt.dim6}.
We are thus left with the following situation:
\begin{lemma}
     Suppose $2a_1+2a_2+a_3+2>0$ and $\Delta_q=0$. 
     Then $-1\notin\sigma(M)$ and the cubic polynomial $q(t)$ in \eqref{cubic.polynomial} has a real multiple root.
    Moreover, the following statements hold.
    \begin{enumerate}
        \item[\rm (i)]
        If $\alpha\beta-9\gamma\geq0$,
        then $M$ has a real symplectic square root.
        
        \item[\rm (ii)]
        If $\alpha\beta-9\gamma<0$,
        Then $M$ has a real symplectic square root if and only if
        \[
        M\sim
        D(\lambda)^{\diamond2}\diamond M_2,
        \]
        where $\lambda<0$ and $M_2\in\operatorname{Sp}(2)$.
        Moreover, $M_2=D(\mu)$ for some $\mu>0$ or
        $M_2=R(\theta)$ for some $\theta\in(0,\pi)\cup(\pi,2\pi)$.
    \end{enumerate}
\end{lemma}
\begin{proof}
By
$\gamma=2a_1+2a_2+a_3+2=p_M(-1)=\det(M+I)>0$,
we have $-1\notin\sigma(M)$.
Since $\Delta_q=0$, the cubic polynomial
$q(t)=t^3-\alpha t^2+\beta t-\gamma$ in \eqref{cubic.polynomial}
has a real multiple root. We divide the proof into two cases.

First suppose that $q(t)$ has a double root $r_1$ and a distinct
simple root $r_2$. Then
\[
q(t)=(t-r_1)^2(t-r_2).
\]
Comparing coefficients gives
\[
\alpha=2r_1+r_2,\qquad
\beta=r_1^2+2r_1r_2,\qquad
\gamma=r_1^2r_2.
\]
Since $\gamma>0$, we have $r_2=\frac{\gamma}{r_1^2}>0$.
Moreover,
\[
\alpha^2-3\beta=(r_1-r_2)^2>0,
\]
and
\[
\alpha\beta-9\gamma=2r_1(r_1-r_2)^2.
\]
Thus the sign of the double root $r_1$ is exactly the sign of
$\alpha\beta-9\gamma$.

The two roots $r_1$ and $r_2$ determine two coprime reciprocal
factors of the characteristic polynomial of $M$:
\[
p_4(t)=\bigl(t^2-(r_1-2)t+1\bigr)^2
\]
and
\[
p_2(t)=t^2-(r_2-2)t+1.
\]
Since $r_1\neq r_2$, the corresponding spectra are disjoint. Hence,
by the symplectic spectral decomposition and Lemma \ref{lem:sqrt.under.decompositions}, 
we may write
\[
M\sim M_4\diamond M_2,
\]
where $M_4\in\operatorname{Sp}(4)$ has characteristic polynomial
$p_4(t)$, and $M_2\in\operatorname{Sp}(2)$ has characteristic
polynomial $p_2(t)$. Moreover, $M$ has a real symplectic square
root if and only if both $M_4$ and $M_2$ have real symplectic
square roots.

Since $r_2>0$, we have
$\operatorname{tr}(M_2)=r_2-2>-2$.
Therefore $M_2$ has a real symplectic square root by
Theorem \ref{thm:sqrt.dim2}.

If $\alpha\beta-9\gamma>0$, then $r_1>0$. 
Denote the corresponding quantities of $M_4$ in Theorem \ref{Thm:Sp4.square.root}
by $\widetilde\alpha,\widetilde\beta$, and $\widetilde\Theta$.
Then we have
\begin{eqnarray*}
\widetilde\alpha&=&2r_1\\
\widetilde\beta&=&r_1^2>0\\
\widetilde\Theta&=&\widetilde\alpha+2\sqrt{\widetilde\beta}=4r_1>0.
\end{eqnarray*}
By Theorem \ref{Thm:Sp4.square.root}, $M_4$ has a real symplectic square root. Since
$M_2$ also has one, Lemma \ref{lem:sqrt.under.decompositions} implies that $M$ has a real
symplectic square root.

If $\alpha\beta-9\gamma<0$, then $r_1<0$. 
In this case
\[
\widetilde\Theta=2r+2|r|=0.
\]
By Theorem \ref{Thm:Sp4.square.root}, the four-dimensional component $M_4$ has a real
symplectic square root if and only if
\[
M_4\sim D(\lambda)\diamond D(\lambda).
\]
Since $M_2$ always has a real symplectic square root, Lemma \ref{lem:sqrt.under.decompositions}
shows that $M$ has a real symplectic square root if and only if
\[
M\sim D(\lambda)^{\diamond2}\diamond M_2,
\]
where $M_2\in\operatorname{Sp}(2)$ satisfies $\operatorname{tr}(M_2)>-2$.
Equivalently, $M_2=D(\mu)$ for some $\mu>0$, or
$M_2=R(\theta)$ for some $\theta\in(0,\pi)\cup(\pi,2\pi)$.
This proves the assertion in the double-root case.

We now consider the case where $q(t)$ has a triple root. Write
\[
q(t)=(t-r)^3.
\]
Then
\[
\gamma=r^3.
\]
Since $\gamma>0$, it follows immediately that $r>0$.
Equivalently, the triple root cannot be negative.
Moreover, $\gamma>0$ implies that $M+I$ is invertible,
and $\sigma_1\sigma_2-\sigma_3=8r^{3/2}>0$.
Therefore the same rational functional construction used in the
proof of Lemma \ref{Lm:2.9} is well defined and gives a real symplectic
matrix $F(M)$ satisfying
\[
F(M)^2=M.
\]
Thus $M$ has a real symplectic square root in the triple-root case.
Note that in such a case, we have
\[
\alpha=3r,\qquad \beta=3r^2,\qquad \gamma=r^3.
\]
Hence
\[
\alpha\beta-9\gamma=9r^3-9r^3=0.
\]
Thus the triple-root case is included in case \({\rm (i)}\).

Combining this with the double-root analysis above, we obtain:
if $\alpha\beta-9\gamma\ge0$,
then $M$ has a real symplectic square root;
if $\alpha\beta-9\gamma<0$,
then $M$ has a real symplectic square root if and only if
\[
M\sim D(\lambda)^{\diamond2}\diamond M_2,
\]
where $\lambda<0$ and $M_2\in\operatorname{Sp}(2)$ satisfies
$\operatorname{tr}(M_2)>-2$.
This completes the proof.
\end{proof}

The preceding lemmas cover the nondegenerate case as well as the
two degenerate cases in which either $-1$ is an eigenvalue or the
cubic polynomial $q(t)$ in \eqref{cubic.polynomial} has multiple roots. We summarize these
results in the following criterion for $6\times6$ symplectic
matrices.
\begin{theorem}
Let \(M\in {\rm Sp}(6)\), and suppose that the characteristic polynomial of 
\(M\) is 
\[ p_M(t)=t^6-a_1t^5+a_2t^4-a_3t^3+a_2t^2-a_1t+1. 
\] 
Set 
\begin{eqnarray*}
\alpha&=&a_1+6,\\
\beta&=&4a_1+a_2+9,\\
\gamma&=&2a_1+2a_2+a_3+2,\\
\Delta_q&=&\alpha^2\beta^2-4\beta^3-4\alpha^3\gamma
-27\gamma^2+18\alpha\beta\gamma.
\end{eqnarray*}
A necessary condition for $M$ to have a real symplectic square root is
$\gamma\ge0$.

Moreover, $M$ has a real symplectic square root if and only if one of
the following mutually exclusive cases holds.

\begin{enumerate}
\item[\rm (i)]
$\gamma>0,\quad
\Delta_q<0$.

\item[\rm (ii)]
$\gamma>0,\quad
\Delta_q>0,\quad
\alpha>0,\quad
\beta>0$.

\item[\rm (iii)]
$\gamma>0,\quad
\Delta_q=0,\quad
\alpha\beta-9\gamma\geq0$.

\item[\rm (iv)]
$\gamma>0,\quad
\Delta_q=0,\quad
\alpha\beta-9\gamma<0$,
and $M\sim D(\lambda)^{\diamond2}\diamond M_2$
for some $\lambda<0$ and some $M_2\in\operatorname{Sp}(2)$,
with $\tr(M_2)>-2$.

\item[\rm (v)]
$\gamma=0,\quad
\beta\ne0$,
and
$M\sim -I_2\diamond M_4$
for some $M_4\in\operatorname{Sp}(4)$ that has a real symplectic square root.


\item[\rm (vi)]
$\gamma=0,\quad
\beta=0,\quad
\alpha>0$,
and
$M\sim
N_1(-1,b)^{\diamond2}\diamond M_2$
for some $b\in\mathbb R$, where $M_2\in\Sp(2)$
with $\tr(M_2)>-2$.

\item[\rm (vii)]
$\gamma=\beta=\alpha=0$,
and either
\[
M\sim
-I_2\diamond N_1(-1,b)^{\diamond2}
\]
for some $b\in\mathbb R$, or
\[
M\sim \begin{pmatrix}
-1&1&0&0&0&s\\
0&-1&0&0&0&0\\
0&0&-1&s&0&0\\
0&0&0&-1&0&0\\
0&0&-2s&1&-1&0\\
0&-2s&0&0&0&-1
\end{pmatrix},
\]
for some $s=\pm1$.
\end{enumerate}

In all remaining cases, $M$ has no real symplectic square root.
\end{theorem}
\begin{remark}
    For Cases (i)-(iii), we can use \eqref{X} to compute a real symplectic square root of $M$.
\end{remark}

\section{Decomposition of Wonenburger matrices}
\label{sec:3}

In this section, we establish the decomposition results for Wonenburger matrices that will be used in the proof of the main theorems. The section is divided into three parts. In Subsection \ref{subsec:3.1}, we first develop general decomposition criteria according to the eigenvalue structure of the matrix $A$ in the Wonenburger representation. These results allow us to separate different spectral components whenever the Wonenburger relations or a suitable $R^*$-action make such a separation possible. Subsection \ref{subsec:3.2} is devoted to the special case in which $M$ has only the eigenvalue $-1$. This case is singled out because the eigenvalue $-1$ is the most delicate one in the square-root problem: as Theorem C indicates, the existence of a real symplectic square root depends on a detailed normal form rather than on a simple spectral condition. Finally, Subsection \ref{subsec:3.3} studies the case $\sigma(M)={\pm\sqrt{-1}}$. This case is closely related to the preceding one, since if a symplectic matrix $M$ has only the eigenvalue $-1$ and admits a real square root $X$, then every eigenvalue of $X$ must belong to ${\pm\sqrt{-1}}$. 
Thus the decomposition in Subsection \ref{subsec:3.3} provides the necessary preparation for analyzing square roots of matrices whose spectrum is concentrated at $-1$.

\subsection{Decomposition by eigenvalues}
\label{subsec:3.1}

The existence problem for a symplectic square root is highly complex and depends critically on the matrix's eigenvalues. In fact, the solvability of this problem differs drastically across different eigenvalues. It is therefore natural and necessary to begin by decomposing the matrix according to its eigenstructure, which allows us to reduce the problem to simpler components characterized by their eigenvalues.

First, we have

\begin{lemma}\label{Lm:decomposition.via.different.eigenvalues}
    In the Wonenburger matrix $M=\begin{pmatrix}A&B\\C&A^T\end{pmatrix}$, if $A$ has the form
    $$A=\left(\begin{matrix}A_1 & O \\O & A_2\end{matrix}\right),$$
    where $A_1,A_2$ are square matrices with no common eigenvalues, i.e., $\sigma(A_1)\cap\sigma(A_2)=\emptyset$,
    then we must have
    $$
    B=\left(\begin{matrix}B_1 & O\\O & B_2\end{matrix}\right),\quad
    C=\left(\begin{matrix}C_1 & O\\O & C_2\end{matrix}\right),
    $$
    where $B_i,C_i$ have the same size as $A_i$ for $i=1,2$, respectively.
    Therefore, we have
    \begin{equation}
        \left(\begin{matrix}A & B\\C & A^T\end{matrix}\right)
        =\left(\begin{matrix}A_1 & B_1\\C_1 & A_1^T\end{matrix}\right)
        \diamond\left(\begin{matrix}A_2 & B_2\\C_2 & A_2^T\end{matrix}\right),
    \end{equation}
    where ``$\diamond$" is defined in \cite{Lon}.
\end{lemma}
\begin{proof}
    Suppose that $A_1\in\mathbb{R}^{n_1\times n_1},A_2\in\mathbb{R}^{n_2\times n_2}$ and
    $$
    B=\left(\begin{matrix}B_{11} & B_{12}\\B_{12}^T & B_{22}\end{matrix}\right),\quad
    C=\left(\begin{matrix}C_{11} & C_{12}\\C_{12}^T & C_{22}\end{matrix}\right),
    $$
    where $B_{11},C_{11}\in\mathbb{R}^{n_1\times n_1}$,
    $B_{12},C_{12}\in\mathbb{R}^{n_1\times n_2}$ and
    $B_{22},C_{22}\in\mathbb{R}^{n_2\times n_2}$.
    Then equation $AB=BA^T$ reads
    $$
    A_1B_{11}=B_{11}A_1^T,\qquad A_2B_{22}=B_{22}A_2^T,
    $$
    and
    \begin{equation}\label{eq.of.B12}
        A_1B_{12}=B_{12}A_2^T.
    \end{equation}

    Now suppose $\lambda\in\sigma(A_1)$, and hence $\lambda\notin\sigma(A_2)$.
    By (\ref{eq.of.B12}), we have
    $$
    (A_1-\lambda I_{n_1})B_{12}=B_{12}(A_2-\lambda I_{n_2})^T.
    $$
    Then we obtain
    $$
    0=(A_1-\lambda I_{n_1})^{n_1}B_{12}=B_{12}((A_2-\lambda I_{n_2})^{n_1})^T,
    $$
    which implies $B_{12}=0$.

    Similarly, we have $C_{12}=0$.
\end{proof}

Then after putting the real matrix $A$ into its real primary decomposition, 
the corresponding Wonenburger matrix splits into a $\diamond$-sum of blocks
associated with individual real eigenvalues or conjugate pairs of non-real eigenvalues.
\begin{corollary}
For any Wonenburger matrix $M=\begin{pmatrix}A&B\\C&A^T\end{pmatrix}$, 
there exists a list of Wonenburger matrices $M_i=\begin{pmatrix}A_i&B_i\\C_i&A_i^T\end{pmatrix},1\le i\le m$, such that
$$
M\sim M_1\diamond M_2\diamond\ldots\diamond M_m,
$$
where, for each $1\leq i\leq m$, the matrix $A_i$ has either a single real eigenvalue or a pair of non-real conjugate eigenvalues. Moreover,
$\sigma(A_i)\cap \sigma(A_j)=\emptyset$ for any $i\neq j$.
\end{corollary}

\medskip

Lemma \ref{Lm:decomposition.via.different.eigenvalues} shows that, when the diagonal blocks of $A$ have disjoint spectra, the Wonenburger relations force the off-diagonal blocks of $B$ and $C$ to vanish automatically. This provides a direct $\diamond$-decomposition of $M$. 
In applications, however, the diagonal blocks of $A$ may have overlapping spectra, so that such a conclusion no longer follows from the Sylvester-type equation alone. The next lemma gives a useful replacement under a nondegeneracy assumption on $B$: although the off-diagonal blocks need not vanish initially, they can be eliminated by a suitable $R^*$-action while preserving the block diagonal form of $A$.

\begin{lemma}\label{Lm:decomposition.by.different.sizes}
    In the Wonenburger matrix $M=\begin{pmatrix}A&B\\C&A^T\end{pmatrix}$, if $A$ has the form
    $$A=\left(\begin{matrix}A_1 & O\\O & A_2\end{matrix}\right),$$
    where $A_1,A_2$ are square matrices.
    Moreover, under the same matrix partitioning scheme as $A$,
    $B$ can be expressed as
    $$
    B=\left(\begin{matrix}B_{11} & B_{12}\\B_{21} & B_{22}\end{matrix}\right).
    $$
    If $|B|\cdot|B_{22}|\ne0$ (or $|B|\cdot|B_{11}|\ne0$),
    then there exists a suitable $R\in{\rm GL}_n(\mathbb{R})$, such that
    under the action with $R$, we have
    \begin{equation}
        R_*\left(\begin{matrix}A & B\\C & A^T\end{matrix}\right)
        =\left(\begin{matrix}A_1 & B_1\\C_1 & A_1^T\end{matrix}\right)
        \diamond\left(\begin{matrix}A_2 & B_2\\C_2 & A_2^T\end{matrix}\right),
    \end{equation}
    for some $B_i,C_i,\;i=1,2,$
    where``$\diamond$" is defined in \cite{Lon}.
\end{lemma}
\begin{proof}
Since $AB=BA^T$, we have
$$
A_1B_{12}=B_{12}A_2^T,\quad
A_2B_{22}=B_{22}A_2^T.
$$
Then we obtain
$$
A_1B_{12}B_{22}^{-1}=B_{12}A_2^TB_{22}^{-1}=B_{12}B_{22}^{-1}A_2.
$$

Now we define
$$
R=\left(\begin{matrix}I & -B_{12}B_{22}^{-1}\\O & I\end{matrix}\right).
$$
Then we have
\begin{eqnarray}
RAR^{-1}&=&\left(\begin{matrix}I & -B_{12}B_{22}^{-1}\\O & I\end{matrix}\right)
\left(\begin{matrix}A_1 & O\\O & A_2\end{matrix}\right)
\left(\begin{matrix}I & B_{12}B_{22}^{-1}\\O & I\end{matrix}\right)
\nonumber\\
&=&\left(\begin{matrix}A_1 & A_1B_{12}B_{22}^{-1}-B_{12}B_{22}^{-1}A_2\\O & A_2\end{matrix}\right)
\nonumber\\
&=&\left(\begin{matrix}A_1 & O\\O & A_2\end{matrix}\right).
\nonumber
\end{eqnarray}
Since $B$ is symmetric, then $B_{11}$ and $B_{22}$ are symmetric
and $B_{21}=B_{12}^T$. Then we obtain
\begin{eqnarray}
RBR^T&=&\left(\begin{matrix}I & -B_{12}B_{22}^{-1}\\O & I\end{matrix}\right)
\left(\begin{matrix}B_{11} & B_{12}\\B_{12}^T & B_{22}\end{matrix}\right)
\left(\begin{matrix}I & O\\-B_{22}^{-1}B_{12}^T & I\end{matrix}\right)
\nonumber\\
&=&\left(\begin{matrix}B_{11}-B_{12}B_{22}^{-1}B_{12}^T & O\\O & B_{22}\end{matrix}\right)
\nonumber
\end{eqnarray}
Finally, we have
\begin{eqnarray}
(R^T)^{-1}CR^{-1}&=&(RBR^T)^{-1}(RBCR^{-1})
\nonumber\\
&=&(RBR^T)^{-1}(R(A^2-I)R^{-1})
\nonumber\\
&=&\left(\begin{matrix}B_{11}-B_{12}B_{22}^{-1}B_{12}^T & O\\O & B_{22}\end{matrix}\right)^{-1}
\left(\begin{matrix}A_1^2-I & O\\O & A_2^2-I\end{matrix}\right)
\nonumber\\
&=&\left(\begin{matrix}[B_{11}-B_{12}B_{22}^{-1}B_{12}^T]^{-1}(A_1^2-I) & O\\O & B_{22}^{-1}(A_2^2-I)\end{matrix}\right),
\end{eqnarray}
where at the last equality, we have $|B_{11}-B_{12}B_{22}^{-1}B_{12}^T|=\frac{|B|}{|B_{22}|}\ne0$.
\end{proof}

\subsection{Decomposition in the case $\sigma(M)=\{-1\}$}
\label{subsec:3.2}

By Lemma 3.2 of \cite{FrM1},
the characteristic polynomial of $M=\begin{pmatrix}A&B\\C&A^T\end{pmatrix}$ is given by
\begin{equation}\label{char.poly}
p_{M}(t)=\det(t^2I_n-2tA+I_n)
=t^np_{-2A}(-t-\frac{1}{t}),
\end{equation}
where $p_{-2A}$ is the characteristic polynomial of $-2A$.
Thus, $\sigma(M)=\{-1\}$ implies $\sigma(A)=\{-1\}$.

For later use, we just need to consider the case when $A$ just have only one Jordan block.
Now, up to a suitable $R^*$-action, we can suppose $A=J_n(-1)$.
\begin{lemma}\label{Lm:one.Jordan.block}
    In Wonenburger matrix $M=\begin{pmatrix}A&B\\C&A^T\end{pmatrix}$, if $A=J_n(-1),n\ge2$, i.e.,
    the Jordan block of eigenvalue $-1$,
    then we must have
    \begin{equation}\label{BC.form.with.one.Jordan.block}
    B=\left(\begin{matrix}
        b_1 & b_2 & \cdots & b_{n-1}& b_n\\
        b_2 & b_3 &\cdots & b_n & 0\\
        \vdots& \vdots& \vdots& \vdots& \vdots\\
        b_{n-1}& b_n & \cdots& 0& 0\\
        b_n& 0& \cdots& 0& 0
    \end{matrix}\right),\quad
    C=\left(\begin{matrix}
        0 & 0 & \cdots & 0& c_n\\
        0 & 0 &\cdots & c_n & c_{n-1}\\
        \vdots& \vdots& \vdots& \vdots& \vdots\\
        0& c_n& \cdots& c_3& c_2\\
        c_n& c_{n-1}& \cdots& c_2& c_1
    \end{matrix}\right),
    \end{equation}
    where $(b_1,b_2,\ldots,b_n)$ and $(c_1,c_2,\ldots,c_n)$
    satisfy
    \begin{equation}\label{eq.of.Cb.n=2}
        C\begin{pmatrix}b_1\\b_2\end{pmatrix}
        =\begin{pmatrix}0\\-2\end{pmatrix},
        \;{\rm if}\;n=2;
    \end{equation}
    and
    \begin{equation}\label{eq.of.Cb.n>2}
        C\begin{pmatrix}b_1\\b_2\\ \vdots\\b_n\end{pmatrix}
        =\begin{pmatrix}0\\-2\\1\\0\\ \vdots\\0\end{pmatrix},
        \;{\rm if}\;n\ge3.
    \end{equation}
\end{lemma}
\begin{remark}
    The equations (\ref{eq.of.Cb.n=2}) and
    (\ref{eq.of.Cb.n>2}) are equivalent to
    \begin{equation}\label{eq.of.Bc.n=2}
        B\begin{pmatrix}c_2\\c_1\end{pmatrix}
        =\begin{pmatrix}-2\\0\end{pmatrix},
        \;{\rm if}\;n=2;
    \end{equation}
    and
    \begin{equation}\label{eq.of.Bc.n>2}
        B\begin{pmatrix}c_n\\c_{n-1}\\ \vdots\\c_1\end{pmatrix}
        =\begin{pmatrix}0\\ \vdots\\0\\1\\-2\\0\end{pmatrix},
        \;{\rm if}\;n\ge3,
    \end{equation}
    respectively.
\end{remark}

Suppose $R\in\mathbb{R}^{n\times n}$ commutes with $A=J_n(-1)$,
i.e., $RJ_n(-1)R^{-1}=J_n(-1)$,
then we must have
\begin{equation}\label{R}
    R=\left(\begin{matrix}
        r_n & r_{n-1} & \cdots & r_2& r_1\\
        0 & r_n &\cdots & r_3 & r_2\\
        \vdots& \vdots& \vdots& \vdots& \vdots\\
        0& 0 & \cdots& r_n& r_{n-1}\\
        0& 0& \cdots& 0& r_n
    \end{matrix}\right),
\end{equation}
for some $(r_1,r_2,\ldots,r_n)\in\mathbb{R}^n$.
Then $R$ is invertible if and only if $r_n\ne0$.
Moreover, we have
$$
RB=BR^T,\qquad, CR=R^TC.
$$

If $A=J_n(-1)$,
under the $R^*$-action in (\ref{R}),
$A$ is invariant, and $B$ becomes
$$
RBR^T=R^2B.
$$
Now we can use this action to simplify the matrices $B,C$ in the normal form $M=\begin{pmatrix}A&B\\C&A^T\end{pmatrix}$.

\begin{lemma}\label{Lm:simple.form.of.one.Jordan.block}
Using the same notations in Lemma \ref{Lm:one.Jordan.block},
    suppose $(b_1,b_2,\ldots,b_n)\ne0$.
    Then there exists an $R\in{\rm GL}_n(\mathbb{R})$ of the form
    (\ref{R}), such that
    $$
    RBR^T=\left(\begin{matrix}
        \tilde{b}_1 & \tilde{b}_2 & \cdots & \tilde{b}_{n-1}& \tilde{b}_n\\
        \tilde{b}_2 & \tilde{b}_3 &\cdots & \tilde{b}_n & 0\\
        \vdots& \vdots& \vdots& \vdots& \vdots\\
        \tilde{b}_{n-1}& \tilde{b}_n & \cdots& 0& 0\\
        \tilde{b}_n& 0& \cdots& 0& 0
    \end{matrix}\right),
    \quad
    (R^T)^{-1}CR^{-1}=\left(\begin{matrix}
        0 & 0 & \cdots & 0& \tilde{c}_n\\
        0 & 0 &\cdots & \tilde{c}_n & \tilde{c}_{n-1}\\
        \vdots& \vdots& \vdots& \vdots& \vdots\\
        0& \tilde{c}_n& \cdots& \tilde{c}_3& \tilde{c}_2\\
        \tilde{c}_n& \tilde{c}_{n-1}& \cdots& \tilde{c}_2& \tilde{c}_1
    \end{matrix}\right),
    $$
    where if $n=2$,
        \begin{equation}
        \begin{pmatrix}\tilde{b}_1\\\tilde{b}_2\end{pmatrix}
        =\begin{pmatrix}0\\\pm1\end{pmatrix},
        \begin{pmatrix}\tilde{c}_1\\\tilde{c}_2\end{pmatrix}
        =\begin{pmatrix}-2\tilde{b}_2\\0\end{pmatrix}
    \end{equation}
    or
    \begin{equation}
        \begin{pmatrix}\tilde{b}_1\\\tilde{b}_2\end{pmatrix}
        =\begin{pmatrix}\pm1\\0\end{pmatrix},
        \begin{pmatrix}\tilde{c}_1\\\tilde{c}_2\end{pmatrix}
        =\begin{pmatrix}0\\-2\tilde{b}_1\end{pmatrix},
    \end{equation}
    and if $n\ge3$,
    \begin{equation}
        \begin{pmatrix}\tilde{b}_1\\\tilde{b}_2\\ \vdots\\\tilde{b}_n\end{pmatrix}
        =\begin{pmatrix}0\\\vdots\\0\\ \pm1\end{pmatrix},
        \begin{pmatrix}\tilde{c}_1\\\tilde{c}_2\\ \vdots\\\tilde{c}_n\end{pmatrix}
        =\begin{pmatrix}0\\\vdots\\0\\ \tilde{b}_n\\-2\tilde{b}_n\\0\end{pmatrix},
    \end{equation}
    or
    \begin{equation}
        \begin{pmatrix}\tilde{b}_1\\\tilde{b}_2\\ \vdots\\\tilde{b}_n\end{pmatrix}
        =\begin{pmatrix}0\\\vdots\\0\\ \pm1\\0\end{pmatrix},
        \begin{pmatrix}\tilde{c}_1\\\tilde{c}_2\\ \vdots\\\tilde{c}_n\end{pmatrix}
        =\begin{pmatrix}0\\\vdots\\0\\ \tilde{b}_{n-1}\\-2\tilde{b}_{n-1}\end{pmatrix}.
    \end{equation}
    
    In particular, we have either
    $$
    {\rm rank}(B)=n,\quad {\rm rank}(C)=n-1,
    $$
    or
    $$
    {\rm rank}(B)=n-1,\quad {\rm rank}(C)=n.
    $$
\end{lemma}
\begin{proof}
    By (\ref{eq.of.Cb.n=2}) or (\ref{eq.of.Cb.n>2}), we have $b_nc_n=0$ and $b_{n-1}c_n+b_nc_{n-1}=-2$,
    which implies $b_n=0$ or $c_n=0$.
    If $b_n=0$, we must have $b_{n-1}\ne0$ and $c_n\ne0$.

    On the other hand, by direct computation, we obtain
\begin{equation}\label{eq.of.R}
        \begin{pmatrix}\tilde{b}_1\\\tilde{b}_2\\ \vdots\\\tilde{b}_n\end{pmatrix}
        =
        \left(\begin{matrix}
        r_n^2 & r_{n-1}r_n+r_nr_{n-1} & \cdots & \sum_{i=2}^{n}r_ir_{n+2-i}& \sum_{i=1}^{n}r_ir_{n+1-i}\\
        0 & r_n^2 &\cdots & \sum_{i=3}^{n}r_ir_{n+3-i} & \sum_{i=2}^{n}r_ir_{n+2-i}\\
        \vdots& \vdots& \vdots& \vdots& \vdots\\
        0& 0 & \cdots& r_n^2& r_{n-1}r_n+r_nr_{n-1}\\
        0& 0& \cdots& 0& r_n^2
    \end{matrix}\right)
        \begin{pmatrix}b_1\\b_2\\\vdots\\b_n\end{pmatrix}.
    \end{equation}
If $b_n\ne0$, in the last equation of (\ref{eq.of.R}),
letting $r_n=\frac{1}{\sqrt{|b_n|}}\ne0$,
we have $\tilde{b}_n={\rm sign}(b_n)$.
By the second last equation of (\ref{eq.of.R}),
$r_{n-1}=-\frac{r_nb_{n-1}}{b_n}$ guarantees $\tilde{b}_{n-1}=0$.
By induction, there also exists $r_1,\ldots,r_{n-2}$ such that
$\tilde{b}_1=\ldots=\tilde{b}_{n-2}=0$.

If $b_n=0$, then $b_{n-1}\ne0$.
Then the remaining argument is similar to the case $b_n\ne0$.

Note that $(\tilde{c}_1,\ldots,\tilde{c}_n)$ is obtained
by (\ref{eq.of.Bc.n=2}) or (\ref{eq.of.Bc.n>2}).
\end{proof}

\subsection{Decomposition in the case $\sigma(M)=\{\pm\sqrt{-1}\}$}
\label{subsec:3.3}

By \eqref{char.poly}, $\sigma(M)=\{\pm\sqrt{-1}\}$ implies $\sigma(A)=\{0\}$.
Therefore, up to a suitable $R^*$-action, we can suppose
\begin{equation}\label{A.Jordan.form}
    A={\rm diag}(\underbrace{J_{k_1}(0),\ldots,J_{k_1}(0)}_{n_1},\underbrace{J_{k_2}(0),\ldots,J_{k_2}(0)}_{n_2},\ldots,\underbrace{J_{k_p}(0),\ldots,J_{k_p}(0)}_{n_p})
\end{equation} 
where $k_1>k_2>\ldots>k_p$.
Since $BC=A^2-I$ and $A$ is nilpotent, the matrix $A^2-I$ is invertible. Therefore both $B$ and $C$ are invertible.

For simplicity, we first consider the case when $A$ just have only one Jordan block, i.e., $A=J_n(0)$.
By a similar argument of Lemma \ref{Lm:one.Jordan.block}, we have
\begin{lemma}\label{Lm:one.Jordan.block.of.eig.0}
    In Wonenburger matrix $M=\begin{pmatrix}A&B\\C&A^T\end{pmatrix}$, if $A=J_n(0),n\ge2$, i.e.,
    the Jordan block of eigenvalue $0$,
    then we must have
    \begin{equation}\label{BC.form.with.one.Jordan.block.of.eig.0}
    B=\left(\begin{matrix}
        b_1 & b_2 & \cdots & b_{n-1}& b_n\\
        b_2 & b_3 &\cdots & b_n & 0\\
        \vdots& \vdots& \vdots& \vdots& \vdots\\
        b_{n-1}& b_n & \cdots& 0& 0\\
        b_n& 0& \cdots& 0& 0
    \end{matrix}\right),\quad
    C=\left(\begin{matrix}
        0 & 0 & \cdots & 0& c_n\\
        0 & 0 &\cdots & c_n & c_{n-1}\\
        \vdots& \vdots& \vdots& \vdots& \vdots\\
        0& c_n& \cdots& c_3& c_2\\
        c_n& c_{n-1}& \cdots& c_2& c_1
    \end{matrix}\right),
    \end{equation}
    where $(b_1,b_2,\ldots,b_n)$ and $(c_1,c_2,\ldots,c_n)$
    satisfy
    \begin{equation}\label{eq.of.Cb.n=2.of.eig.0}
        C\begin{pmatrix}b_1\\b_2\end{pmatrix}
        =\begin{pmatrix}-1\\0\end{pmatrix},
        \;{\rm if}\;n=2;
    \end{equation}
    and
    \begin{equation}\label{eq.of.Cb.n>2.of.eig.0}
        C\begin{pmatrix}b_1\\b_2\\ \vdots\\b_n\end{pmatrix}
        =\begin{pmatrix}-1\\0\\1\\0\\ \vdots\\0\end{pmatrix},
        \;{\rm if}\;n\ge3.
    \end{equation}
    In particular, we have
    $$
    {\rm rank}(B)={\rm rank}(C)=n.
    $$
\end{lemma}

Now we study the general case in which $A$ has multiple Jordan blocks, as shown in \eqref{A.Jordan.form}.
Let $m=n_1+\cdots+n_p$ be the total number of Jordan blocks of $A$, and write
\[
B=\left(\begin{matrix}
        B_{11} & B_{12} & \cdots & B_{1m}\\
        B_{21} & B_{22} &\cdots & B_{2m}\\
        \vdots& \vdots& \vdots& \vdots\\
        B_{m1}& B_{m2} & \cdots& B_{mm}
    \end{matrix}\right),
\]
according to this Jordan block decomposition. If the $i$-th and $j$-th Jordan blocks have sizes $\kappa_i$ and $\kappa_j$, respectively, then the equation
$AB=BA^T$
gives
\[
J_{\kappa_i}(0)B_{ij}=B_{ij}J_{\kappa_j}(0)^T.
\]
Thus
$B_{ij}$ is a rectangular Hankel-type matrix with $\min\{\kappa_i,\kappa_j\}$ free parameters, and the remaining entries are forced to be zero.
That is,
\begin{equation}
B_{ij}=
    \left(\begin{matrix}
        b_1^{(ij)} & b_2^{(ij)} & \cdots & b_{k-1}^{(ij)}& b_k^{(ij)}\\
        b_2^{(ij)} & b_3^{(ij)} &\cdots & b_k^{(ij)} & 0\\
        \vdots& \vdots& \vdots& \vdots& \vdots\\
        b_{k-1}^{(ij)}& b_k^{(ij)} & \cdots& 0& 0\\
        b_k^{(ij)}& 0& \cdots& 0& 0\\
        0& 0& \cdots& 0& 0\\
        \vdots& \vdots& \vdots& \vdots& \vdots\\
        0& 0& \cdots& 0& 0
    \end{matrix}\right),\quad\text{for}\;\kappa_i\le\kappa_j,
\end{equation}
or
\begin{equation}
B_{ij}=
    \left(\begin{matrix}
        b_1^{(ij)} & b_2^{(ij)} & \cdots & b_{k-1}^{(ij)}& b_k^{(ij)}& 0& \cdots& 0\\
        b_2^{(ij)} & b_3^{(ij)} &\cdots & b_k^{(ij)} & 0& 0& \cdots& 0\\
        \vdots& \vdots& \vdots& \vdots& \vdots& 0& \cdots& 0\\
        b_{k-1}^{(ij)}& b_k^{(ij)} & \cdots& 0& 0& 0& \cdots& 0\\
        b_k^{(ij)}& 0& \cdots& 0& 0& 0& \cdots& 0
    \end{matrix}\right),\quad{\text{for}\;}\kappa_i>\kappa_j,
\end{equation}
where $k=\min\{\kappa_i,\kappa_j\}$.
In particular, if $i$ and $j$ belong to the same group of Jordan blocks of size $k$, then the last anti-diagonal coefficient $b_k^{(ij)}$ determines the leading contribution to the determinant.

For $1\leq \alpha\leq p$, let $B_\alpha$ denote the submatrix of $B$ corresponding to all Jordan blocks of size $k_\alpha$, and define
\[
\Gamma_\alpha=
\begin{pmatrix}
        b_{k_\alpha}^{(m_{\alpha-1}+1,m_{\alpha-1}+1)} & b_{k_\alpha}^{(m_{\alpha-1}+1,m_{\alpha-1}+2)} & \cdots & b_{k_\alpha}^{(m_{\alpha-1}+1,m_{\alpha-1}+n_\alpha)}\\
        b_{k_\alpha}^{(m_{\alpha-1}+2,m_{\alpha-1}+1)} & b_{k_\alpha}^{(m_{\alpha-1}+2,m_{\alpha-1}+2)} &\cdots & b_{k_\alpha}^{(m_{\alpha-1}+2,m_{\alpha-1}+n_\alpha)}\\
        \vdots& \vdots& \vdots& \vdots\\
        b_{k_\alpha}^{(m_{\alpha-1}+n_\alpha,m_{\alpha-1}+1)}& b_{k_\alpha}^{(m_{\alpha-1}+n_\alpha,m_{\alpha-1}+2)} & \cdots& b_{k_\alpha}^{(m_{\alpha-1}+n_\alpha,m_{\alpha-1}+n_\alpha)}
    \end{pmatrix},
\qquad
m_\alpha=n_1+\cdots+n_\alpha.
\]
In short, we write $\Gamma_\alpha=(b_{k_\alpha}^{(i,j)}),\;m_{\alpha-1}<i,j\le m_\alpha$.

Let
\begin{equation}
    K={\rm diag}(\underbrace{K_{k_1},\ldots,K_{k_1}}_{n_1},\underbrace{K_{k_2},\ldots,K_{k_2}}_{n_2},\ldots,\underbrace{K_{k_p},\ldots,K_{k_p}}_{n_p})
    \nonumber
\end{equation} 
where $K_{k_i}$ is the $k_i\times k_i$ anti-diagonal matrix
for $1\le i\le p$,
which is given by \eqref{K}. 
Since
\[
J_{k_i}(0)^T K_{k_i}=K_{k_i}J_{k_i}(0),
\]
we have $A^T K=KA$.
It follows from $AB=BA^T$ that
\[
A(BK)=BA^T K=BKA.
\]
Thus $T:=BK$
commutes with the Jordan matrix $A$. 
By Lemma 3.5 and Theorem 5.1 of \cite{FMM}, we have
\begin{eqnarray}
\det B_\alpha K_\alpha&=&(\det\Gamma_\alpha)^{k_\alpha},\;1\le\alpha\le p,
\nonumber\\
\det BK&=&\prod_{\alpha=1}^p\det B_\alpha K_\alpha.
\nonumber
\end{eqnarray}
Moreover, since $\det K=\pm1$ and $\det K_\alpha=\pm1$ for $1\le\alpha\le p$, 
we have
\begin{eqnarray}
\det B_\alpha&=&\pm(\det\Gamma_\alpha)^{k_\alpha},
\\
\det B&=&\pm\prod_{\alpha=1}^p\det B_\alpha=\pm\prod_{\alpha=1}^p \bigl(\det \Gamma_\alpha\bigr)^{k_\alpha}.
\end{eqnarray}


Since $B$ is invertible, each $\Gamma_\alpha$ is invertible. Equivalently, each diagonal group $B_\alpha$ is nondegenerate in the sense required for applying Lemma \ref{Lm:decomposition.by.different.sizes}.
Applying Lemma \ref{Lm:decomposition.by.different.sizes} successively to the decomposition corresponding to the groups of dimensions $k_1n_1,k_2n_2,\ldots,k_pn_p$, we separate the groups of Jordan blocks with different sizes.
Thus, there exists $R\in{\rm GL}_n(\mathbb{R})$ such that
\[
R_*\begin{pmatrix}A&B\\C&A^T\end{pmatrix}=
\begin{pmatrix}
A_1 & \overline B_1\\
\overline C_1 & A_1^T
\end{pmatrix}
\diamond
\begin{pmatrix}
A_2 & \overline B_2\\
\overline C_2 & A_2^T
\end{pmatrix}
\diamond\cdots\diamond
\begin{pmatrix}
A_p & \overline B_p\\
\overline C_p & A_p^T
\end{pmatrix},
\]
where
$A_\alpha=\operatorname{diag}(\underbrace{J_{k_\alpha}(0),\ldots,J_{k_\alpha}(0)}_{n_\alpha})$ for $1\leq \alpha\leq p$.

Therefore, in order to complete the decomposition, it remains to consider the case in which $A$ consists of several Jordan blocks all having the same size. This is the content of the next lemma.

\begin{lemma}
If $A={\rm diag}(\underbrace{J_{k}(0),\ldots,J_{k}(0)}_{q})$,
then there exists an invertible matrix $R=\begin{pmatrix}R_{11}& \ldots& R_{1q}\\ \vdots& &\vdots\\R_{q1}& \ldots&R_{qq}\end{pmatrix}$,  
partitioned in the same way as $A$, such that
\begin{eqnarray}
RBR^T
&=&{\rm diag}(\widetilde{B}_{1},\widetilde{B}_{2},\ldots,\widetilde{B}_{q}),\label{tilde.B}\\
(R^T)^{-1}CR^{-1}
&=&{\rm diag}(\widetilde{C}_{1},\widetilde{C}_{2},\ldots,\widetilde{C}_{q}).\label{tilde.C}
\end{eqnarray}
Therefore, under the action by $R$, we have
\begin{equation}
    R_*\begin{pmatrix}A&B\\C&A^T\end{pmatrix}
    =\begin{pmatrix}J_k(0)&\tilde{B}_{1}\\\tilde{C}_{1}&J_k(0)^T\end{pmatrix}
    \diamond\ldots\diamond\begin{pmatrix}J_k(0)&\tilde{B}_{q}\\\tilde{C}_{q}&J_k(0)^T\end{pmatrix}.
\end{equation}
\end{lemma}
\begin{proof}
We first introduce some notation.
With the same partitioning as $A$, suppose
$$
B=\left(\begin{matrix}
        B_{11} & B_{12} & \cdots & B_{1q}\\
        B_{12}^T & B_{22} &\cdots & B_{2q}\\
        \vdots& \vdots& \vdots& \vdots\\
        B_{1q}^T& B_{2m}^T & \cdots& B_{qq}
    \end{matrix}\right),
\quad
C=\left(\begin{matrix}
        C_{11} & C_{12} & \cdots & C_{1q}\\
        C_{12}^T & C_{22} &\cdots & C_{2q}\\
        \vdots& \vdots& \vdots& \vdots\\
        C_{1q}^T& C_{2m}^T & \cdots& C_{qq}
    \end{matrix}\right).
$$
Then by a similar argument to the one in the proof of Lemma \ref{Lm:one.Jordan.block.of.eig.0},
we must have
\begin{equation}\label{B_ij.C_ij.form}
B_{ij}=
    \left(\begin{matrix}
        b_1^{(ij)} & b_2^{(ij)} & \cdots & b_{k-1}^{(ij)}& b_k^{(ij)}\\
        b_2^{(ij)} & b_3^{(ij)} &\cdots & b_k^{(ij)} & 0\\
        \vdots& \vdots& \vdots& \vdots& \vdots\\
        b_{k-1}^{(ij)}& b_k^{(ij)} & \cdots& 0& 0\\
        b_k^{(ij)}& 0& \cdots& 0& 0
    \end{matrix}\right),\;\;
C_{ij}=\left(\begin{matrix}
        0 & 0 & \cdots & 0& c_k^{(ij)}\\
        0 & 0 &\cdots & c_k^{(ij)} & c_{k-1}^{(ij)}\\
        \vdots& \vdots& \vdots& \vdots& \vdots\\
        0& c_k^{(ij)}& \cdots& c_3^{(ij)}& c_2^{(ij)}\\
        c_k^{(ij)}& c_{k-1}^{(ij)}& \cdots& c_2^{(ij)}& c_1^{(ij)}
    \end{matrix}\right),\quad
    1\le i\le j\le q,
\end{equation}
for some $(b_1^{(ij)},b_2^{(ij)},\ldots,b_k^{(ij)})$ and $(c_1^{(ij)},c_2^{(ij)},\ldots,c_k^{(ij)})$.

Suppose $R=\begin{pmatrix}R_{11}& \ldots& R_{1q}\\ \vdots& &\vdots\\R_{q1}& \ldots&R_{qq}\end{pmatrix}\in\mathbb{R}^{qk\times qk}$ commutes with $A$,
i.e., $RAR^{-1}=A$,
then we must have
\begin{equation}\label{R_ij}
    R_{ij}=\left(\begin{matrix}
        r_k^{(ij)} & r_{k-1}^{(ij)} & \cdots & r_2^{(ij)}& r_1^{(ij)}\\
        0 & r_k^{(ij)} &\cdots & r_3^{(ij)} & r_2^{(ij)}\\
        \vdots& \vdots& \vdots& \vdots& \vdots\\
        0& 0 & \cdots& r_k^{(ij)}& r_{k-1}^{(ij)}\\
        0& 0& \cdots& 0& r_k^{(ij)}
    \end{matrix}\right),
    \quad 1\le i, j\le q,
\end{equation}
for some $(r_1^{(ij)},r_2^{(ij)},\ldots,r_k^{(ij)})\in\mathbb{R}^k$.
Moreover, we have
$$
R_{ij}B_{i'j'}=B_{i'j'}R_{ij}^T,\quad, C_{i'j'}R_{ij}=R_{ij}^TC_{i'j'},\quad
R_{ij}R_{i'j'}=R_{i'j'}R_{ij},
$$
for some $1\le i, j\le q$ and $1\le i'\le j'\le q$.

Let
$$
R_i=\begin{pmatrix}
    r_i^{(11)}&\ldots& r_i^{(1q)}\\
    \vdots& &\vdots\\
    r_i^{(q1)}&\ldots& r_i^{(qq)}
\end{pmatrix},
\quad 1\le i\le k.
$$
Then by Lemma 3.5 of \cite{FMM}, we have
$$
\det(R)
=\left|
\begin{matrix}
    r_k^{(11)}&\ldots& r_k^{(1q)}\\
    \vdots& &\vdots\\
    r_k^{(q1)}&\ldots& r_k^{(qq)}
\end{matrix}
\right|^k=|R_k|^k.
$$
Hence $R$ is invertible if and only if $R_k$ is invertible.

Now we give the proofs.
Partition $\widetilde{B}=RBR^T$ in the same way as $B$, and denote the
block in the $i$-th row and $j$-th column by $\widetilde{B}_{ij}$ for $1\le i.j\le q$.
It is sufficient to prove $\widetilde{B}_{ij}=0$ for all $1\le i< j\le q$ for some suitable matrix $R$.
In fact, $(R^T)^{-1}CR^{-1}={\widetilde{B}}^{-1}((RAR^{-1})^2-I)$ implies that $(R^T)^{-1}CR^{-1}$ is also block diagonal.

A direct computation yields
$$
\widetilde{B}_{ij}=\sum_{s,t=1}^{q}R_{is}B_{st}R_{jt}^T,
\quad 1\le i.j\le q.
$$
Hence, we have
\begin{equation}\label{bij.k}
\tilde{b}_k^{(ij)}=\sum_{s,t=1}^{q}{r}_k^{(is)}{b}_k^{(st)}{r}_k^{(jt)}
\end{equation}
and
\begin{equation}\label{bij.k-1}
\tilde{b}_{k-l}^{(ij)}=\sum_{s,t=1}^{q}\sum_{l_1+l_2+l_3=l,l_1,l_2,l_3\ge0}{r}_{k-l_1}^{(is)}{b}_{k-l_2}^{(st)}{r}_{k-l_3}^{(jt)},
\qquad 1\le l\le k-1.
\end{equation}

Define
$$
B_i=\begin{pmatrix}
    b_i^{(11)}& \ldots&b_i^{(1q)}\\
    \vdots& &\vdots\\ 
    b_i^{(q1)}& \ldots& b_i^{(qq)}
\end{pmatrix},
\widetilde{B}_i=
\begin{pmatrix}
    \tilde{b}_i^{(11)}&\ldots& \tilde{b}_i^{(1q)}\\ 
    \vdots& & \vdots\\
    \tilde{b}_i^{(q1)}& \ldots& \tilde{b}_i^{(qq)}
\end{pmatrix},
\quad 1\le i\le k,
$$
by (\ref{bij.k}) and (\ref{bij.k-1}),
we have
\begin{eqnarray}
R_kB_kR_k^T&=&\widetilde{B}_k,\label{eq.of.Bk}\\
R_kB_{k-1}R_k^T+R_{k-1}B_kR_k^T+R_kB_kR_{k-1}^T&=&\widetilde{B}_{k-1}.\label{eq.of.Bk-1}
\end{eqnarray}
and
\begin{equation}\label{eq.of.Bk-l}
R_{k-l}B_kR_k^T+R_kB_kR_{k-l}^T+\sum_{0\le l_1,l_3\le l-1,\;l_1+l_3\le l}R_{k-l_1}B_{k-(l-l_1-l_3)}R_{k-l_3}^T=\widetilde{B}_{k-l},\quad 2\le l\le k-1.
\end{equation}
Since $B_k$ is symmetric, by (\ref{eq.of.Bk}),
there exists an invertible matrix $R_k$
such that $\widetilde{B}_k$ is diagonal.
Note that $R_k$ is invertible implies that $R$ is also invertible.
Furthermore, $BC=A^2-I$ implies that $B$ is invertible.
Together with $|B|=|B_k|^k$, it follows that $B_k$ is also invertible. 

Therefore, 
by a real congruence transformation we may choose an invertible matrix $R_k$ such that $\widetilde B_k=R_kB_kR_k^T$ is diagonal and invertible. 
Moreover, by rescaling the diagonal entries if necessary, we may assume that the sum of any two diagonal entries of $\widetilde B_k$ is nonzero.

Plugging (\ref{eq.of.Bk}) into (\ref{eq.of.Bk-1}), we have
$$
(R_{k-1}R_k^{-1})\widetilde{B}_k+\widetilde{B}_k(R_{k-1}R_k^{-1})^T=\widetilde{B}_{k-1}-R_kB_{k-1}R_k^T.
$$
Since the sum of any two eigenvalues of $\widetilde{B}_k$ is nonzero,
the matrix $I\otimes\widetilde{B}_k+\widetilde{B}_k\otimes I$ is invertible.
Therefore, the matrix equation
$$
X\tilde{B}_k+\tilde{B}_kX^T=-R_kB_{k-1}R_k^T
$$
has a unique symmetric solution $X_{k-1}$. 
Taking $R_{k-1}=X_{k-1}R_k$, we obtain 
$$\widetilde B_{k-1}=0.$$
Again, since $\widetilde B_k$ is diagonal and the sum of any two of its diagonal entries is nonzero, the matrix equation
\[
X_{k-l}\widetilde B_k+\widetilde B_kX_{k-l}^T=-S_l,
\]
where $S_l=\sum_{0\le l_1,l_3\le l-1,\;l_1+l_3\le l}R_{k-l_1}B_{k-(l-l_1-l_3)}R_{k-l_3}^T,\;2\le l\le k-1$,
has a unique symmetric solution $X_{k-l}$. 
Taking $R_{k-l}=X_{k-l}R_k$, by \eqref{eq.of.Bk-l}, we obtain $\widetilde B_{k-l}=0$.

By induction, we can choose $R_{k-1},R_{k-2},\ldots,R_1$ such that
\[
\widetilde B_{k-1}=\widetilde B_{k-2}=\cdots=\widetilde B_1=0.
\]
Therefore, the only possibly nonzero coefficient matrix is $\widetilde B_k$, which is diagonal. Hence all off-diagonal blocks of $\widetilde{B}=RB R^T$ vanish, that is,
\[
\widetilde{B}=RBR^T=\operatorname{diag}(\tilde b_k^{(11)}K_k,\tilde b_k^{(22)}K_k,\ldots,\tilde b_k^{(qq)}K_k)
\]
in the block sense. 
Here $K_k$ is given by \eqref{K}.

At last, since $RAR^{-1}=A$ and $A^2-I$ is block diagonal, it follows from
\[
(R^T)^{-1}CR^{-1}=(RB R^T)^{-1}(A^2-I)
\]
that $(R^T)^{-1}CR^{-1}$ is also block diagonal. This proves the lemma.
\end{proof}

Based on the above argument, we have the following theorem.
\begin{theorem}\label{Thm:decomposition.of.eig.i}
    Suppose $M=\begin{pmatrix}A&B\\C&A^T\end{pmatrix}\in\Sp(2n)$ only possesses eigenvalues $\pm\sqrt{-1}$.
    Then there exists an invertible matrix $R\in{\rm GL}_n(\mathbb{R})$, such that under the action by $R$,
    we have
    \begin{equation}
    R_*\begin{pmatrix}A&B\\C&A^T\end{pmatrix}
    =\begin{pmatrix}A_1&{B}_{1}\\{C}_{1}&A_1^T\end{pmatrix}
    \diamond\ldots\diamond\begin{pmatrix}A_m&{B}_{m}\\{C}_{m}&A_m^T\end{pmatrix},
\end{equation}
where each $A_i,\;1\le i\le m$ only has one Jordan block.
For $A_i=J_{k_i}(0),k_i\ge1$,
    we must have
    \begin{equation}
    B_i=\left(\begin{matrix}
        b_1^{(i)} & b_2^{(i)} & \cdots & b_{k_i-1}^{(i)}& b_{k_i}^{(i)}\\
        b_2^{(i)} & b_3^{(i)} &\cdots & b_{k_i}^{(i)} & 0\\
        \vdots& \vdots& \vdots& \vdots& \vdots\\
        b_{k_i-1}^{(i)}& b_{k_i}^{(i)} & \cdots& 0& 0\\
        b_{k_i}^{(i)}& 0& \cdots& 0& 0
    \end{matrix}\right),\quad
    C_i=\left(\begin{matrix}
        0 & 0 & \cdots & 0& c_{k_i}^{(i)}\\
        0 & 0 &\cdots & c_{k_i}^{(i)} & c_{k_i-1}^{(i)}\\
        \vdots& \vdots& \vdots& \vdots& \vdots\\
        0& c_{k_i}^{(i)}& \cdots& c_3^{(i)}& c_2^{(i)}\\
        c_{k_i}^{(i)}& c_{k_i-1}^{(i)}& \cdots& c_2^{(i)}& c_1^{(i)}
    \end{matrix}\right),
    \end{equation}
    where $(b_1^{(i)},b_2^{(i)},\ldots,b_{k_i}^{(i)})$ and $(c_1^{(i)},c_2^{(i)},\ldots,c_{k_i}^{(i)})$
    satisfy
    \begin{eqnarray}
        b_1^{(i)}c_1^{(i)}&=&-1,\;{\rm if}\;k_i=1;
        \nonumber\\
        C_i\begin{pmatrix}b_1^{(i)}\\b_2^{(i)}\end{pmatrix}
        &=&\begin{pmatrix}-1\\0\end{pmatrix},
        \;{\rm if}\;k_i=2;
    \end{eqnarray}
    and
    \begin{equation}
        C_i\begin{pmatrix}b_1^{(i)}\\b_2^{(i)}\\ \vdots\\b_{k_i}^{(i)}\end{pmatrix}
        =\begin{pmatrix}-1\\0\\ \vdots\\0\\1\\0\\ \vdots\\0\end{pmatrix},
        \;{\rm if}\;k_i\ge3,
    \end{equation}
    where the element $1$ is located at the $(\lfloor\frac{k_i-1}{2}\rfloor+2)$-th row of the right hand side.
Moreover, for each $M_i=\begin{pmatrix}A_i&B_i\\C_i&A_i^T\end{pmatrix},\;1\le i\le m$,
there exists an $R_i\in{\rm GL}_{k_i}(\mathbb{R})$ of the form
    (\ref{R}), such that
    $$
    R_iA_iR_i^{-1}=A_i,\quad
    R_iB_iR_i^T=\pm K_{k_i},
    $$
where $K_{k_i}$ is the $k_i\times k_i$ matrix whose anti-diagonal entries all equal to $1$ and whose other entries are $0$, i.e.,

    $$
    K_{k_i}=\left(\begin{matrix}
        0 & 0 & \cdots & 0& 1\\
        0 & 0 &\cdots & 1 & 0\\
        \vdots& \vdots& \begin{sideways}$\ddots$\end{sideways}& \vdots& \vdots\\
        0& 1 & \cdots& 0& 0\\
        1& 0& \cdots& 0& 0
    \end{matrix}\right).
    $$
\end{theorem}
\begin{remark}
    The decomposition in the above theorem fails when $M$ only possesses eigenvalue $-1$.
    The reason is that, the condition $|B|\ne0$ is required
    for Lemma \ref{Lm:decomposition.by.different.sizes}, which may fail for eigenvalue $-1$.
\end{remark}

When we consider the square of $M$ which only possesses eigenvalue $\pm\sqrt{-1}$, we have
\begin{lemma}\label{Lm:square.of.eig.i}
    Suppose $M=\begin{pmatrix}J_n(0)&sK_n\\C&J_n(0)^T\end{pmatrix}\in\Sp(2n)$ where $s=1$ or $-1$.
    If $n=2m$,
    then there exists an invertible matrix $R\in{\rm GL}_n(\mathbb{R})$, such that under the $R_*$-action,
    we have
    \begin{equation}
    R_*(M^2)=
    \begin{pmatrix}J_m(-1)&B_1\\-sK_m&J_m(-1)^T\end{pmatrix}
    \diamond\begin{pmatrix}J_m(-1)&sK_m\\C_2&J_m(-1)^T\end{pmatrix},
\end{equation}
which is symplectically similar to $\begin{pmatrix}J_m(-1)&sK_m\\C_2&J_m(-1)^T\end{pmatrix}^{\diamond2}$.

If $n=2l+1$,
then there exists an invertible matrix $R\in{\rm GL}_n(\mathbb{R})$, such that under the $R_*$-action,
    we have
    \begin{equation}
    R_*(M^2)=\begin{pmatrix}A_1&B_1\\C_1&A_1^T\end{pmatrix},
\end{equation}
where
$$
A_1={\rm diag}(J_{l+1}(-1),J_l(-1))
$$
and 
$$
B_1=\begin{pmatrix}
                        & & sK_l\\
                        & 0& &\\
                        sK_l& & &
                    \end{pmatrix}.
$$
\end{lemma}
\begin{remark}
    For odd $n$, the matrix $M^2$ is not symplectically similar to any symplectic matrix that is decomposable as a $\diamond$-sum.
\end{remark}
\begin{proof}
    By direct computation and using (\ref{constraints.of.Wonenburger.form}), we have
$$
M^2=\begin{pmatrix}-I_n+2J_n(0)^2&2sJ_n(0){K}_n\\2{C}J_n(0)&[-I_n+2J_n(0)^2]^T\end{pmatrix}.
$$

Let
$$
Q_1={\rm diag}(2^{\frac{1}{2}},2^{\frac{2}{2}},\ldots,2^{\frac{n-1}{2}},2^{\frac{n}{2}}).
$$
Then we obtain
\begin{eqnarray}
    &&Q_1[-I_n+2J_n(0)^2]Q_1^{-1}
    \nonumber\\
    &&\quad=
    \begin{pmatrix}
        2^{\frac{1}{2}}& & & \\& 2^{\frac{2}{2}}\\
        & & \ddots &\\
        & & & 2^{\frac{n}{2}}
    \end{pmatrix}
    \begin{pmatrix}
        -1& 0& 2& & & \\
        & -1 & 0& 2& & \\
        & & \ddots& \ddots& \ddots& \\
        & & & -1& 0& 2\\
        & & & & -1& 0\\
        & & & & & -1
    \end{pmatrix}
    \begin{pmatrix}
        2^{-\frac{1}{2}}& & & \\& 2^{-\frac{2}{2}}\\
        & & \ddots &\\
        & & & 2^{-\frac{n}{2}}
    \end{pmatrix}
    \nonumber\\
    &&\quad=\begin{pmatrix}
        -1& 0& 1& & & \\
        & -1 & 0& 1& & \\
        & & \ddots& \ddots& \ddots& \\
        & & & -1& 0& 1\\
        & & & & -1& 0\\
        & & & & & -1
    \end{pmatrix}
    \nonumber\\
    &&\quad=-I_n+J_n(0)^2,
    \label{Q1.action}
\end{eqnarray}
and
\begin{eqnarray}
    &&Q_1[2sJ_n(0){K}_n]Q_1^T
    \nonumber\\
    &&\quad=
    \begin{pmatrix}
        2^{\frac{1}{2}}& & & \\& 2^{\frac{2}{2}}\\
        & & \ddots &\\
        & & & 2^{\frac{n}{2}}
    \end{pmatrix}
    \begin{pmatrix}
        & & 2s& 0& \medskip\\ 
        & \begin{sideways}$\ddots$\end{sideways}& \begin{sideways}$\ddots$\end{sideways}& &  \medskip\\

        2s& 0& & & \medskip\\
        0& & & & \medskip\\
        & & & & 
    \end{pmatrix}
    \begin{pmatrix}
        2^{\frac{1}{2}}& & & \\& 2^{\frac{2}{2}}\\
        & & \ddots &\\
        & & & 2^{\frac{n}{2}}
    \end{pmatrix}
    \nonumber\\
    &&\quad=2^{\frac{n}{2}+1}s\begin{pmatrix}
        & & 1& 0& \medskip\\ 
        & \begin{sideways}$\ddots$\end{sideways}& \begin{sideways}$\ddots$\end{sideways}& &  \medskip\\

        1& 0& & & \medskip\\
        0& & & & \medskip\\
        & & & & 
    \end{pmatrix}
    \nonumber\\
    &&\quad=2^{\frac{n}{2}+1}sJ_n(0)K_n.
\end{eqnarray}

Now we will construct matrices $R\in{\rm GL}_n(\mathbb{R})$
which satisfies $R[-I_n+2J_n(0)^2]R^{-1}={\rm diag}(J_{\lfloor\frac{n+1}{2}\rfloor}(-1),J_{\lfloor\frac{n-1}{2}\rfloor}(-1))$.
We will discuss the cases based on whether $n$ is even or odd.

\medskip
{\bf Case 1: $n=2m$ is even.}

Let
$$
Q_2=\{q_{ij}\}_{2m\times2m},
$$
which entries satisfy
$$
q_{i,2i}=q_{m+i,2i-1}=1,\quad\forall 1\le i\le m,
$$
and all other entries of $Q_2$ equal to $0$.
Then we have
$$
{Q_2}^{-1}={Q_2}^T.
$$
We denote the entries of $J_n(0)$ and $Q_2J_n(0)^2Q_2^{-1}$ by $a_{ij}$ and $\tilde{a}_{ij}$, respectively, for $1\le i,j\le n$.
Then we have
\begin{equation}\label{Q2.action}
\tilde{a}_{ij}=\sum_{r,s,t=1}^{2m}q_{ir}a_{rs}a_{st}q_{jt}.
\end{equation}
We fix an index $i$. 

If $1 \le i \le m-1$, then by the definition of $Q_2$, the only nonzero terms on the right-hand side of (\ref{Q2.action}) occur when $r=2i$.
Similarly, if we consider the possible nonzero terms on the right-hand side of (\ref{Q2.action}) , there must hold
$s=r+1=2i+1,t=s+1=2i+2$ and $j=t/2=i+1$.
Thus we have $\tilde{a}_{i,i+1}=q_{i,2i}a_{2i,2i+1}a_{2i+1,2i+2}q_{i+1,2i+2}=1$
which is the unique nonzero entry in the $i$-th row of $Q_2J_n(0)^2Q_2^{-1}$.

If $i=m$, then $q_{ir}\ne0$ implies $r=2i=2m$.
But $a_{ms}=0$ for all $1\le s\le 2m$. Thus all terms in the
sum of (\ref{Q2.action}) are zero.
Thus we have $\tilde{a}_{mj}=0$ for all $1\le j\le 2m$.

If $m+1 \le i \le 2m-1$, nonzero terms in (\ref{Q2.action})
require $r=2(i-m)-1,s=r+1=2(i-m),t=s+1=2(i-m)+1$ and $j=(t+1)/2+m=i+1$.
Thus $\tilde{a}_{ij}=1$ is the unique nonzero entry in the $i$-th row of $Q_2J_n(0)^2Q_2^{-1}$.

If $i=2m$, nonzero terms in (\ref{Q2.action})
require $r=2m-1,s=2m,t=2m+1$ which is impossible.
Thus we have $\tilde{a}_{2m,j}=0$ for all $1\le j\le 2m$.

Based on the above argument, we obtain
\begin{eqnarray}
Q_2[Q_1(-I_{2m}+2J_{2m}(0)^2)Q_1^{-1}]Q_2^{-1}
&=&Q_2[-I_{2m}+J_{2m}(0)^2]Q_2^T
\nonumber\\
&=&-I_{2m}+Q_2J_{2m}(0)^2Q_2^{-1}
\nonumber\\
&=&-I_{2m}+{\rm diag}(J_m(0),J_m(0))
\nonumber\\
&=&{\rm diag}(J_m(-1),J_m(-1)).
\end{eqnarray}
Similarly, we have
\begin{eqnarray}
Q_2[Q_1(2sJ_{2m}(0)K_{2m})Q_1^{T}]Q_2^T
&=&Q_2[2^{\frac{n}{2}+1}sJ_{2m}(0)K_{2m}]Q_2^T
\nonumber\\
&=&2^{\frac{n}{2}+1}{\rm diag}(sK_{m-1},0,sK_m).
\end{eqnarray}

Suppose $\bar{B}={\rm diag}(K_{m-1},0)$
and
\begin{equation}\label{relation.C.C22}
(Q_2^{-1})^T(Q_1^{-1})^T[CJ_{2m}(0)]Q_1^{-1}Q_2^{-1}
=\begin{pmatrix}
    C_{11}& C_{12}\\ C_{12}^T& C_{22}
\end{pmatrix}.
\end{equation}
Then we have
$$2^{m+1}\begin{pmatrix}
    \bar{B}& O\\ O& K_{m}
\end{pmatrix}\begin{pmatrix}
    C_{11}& C_{12}\\ C_{12}^T& C_{22}
\end{pmatrix}=
2^{m+1}\begin{pmatrix}
    -I_m+J_m(-1)^2& O\\ O& -I_{m}+J_m(-1)^2
\end{pmatrix},
$$
which implies
$$
K_mC_{12}^T=0.
$$
Since $|K_m|=(-1)^{\frac{m(m-1)}{2}}\ne0$,
there must hold $C_{12}=0$.
Furthermore, by letting
$$
Q_3=2^{-\frac{m+1}{2}}I_{2m},
$$
we obtain
$$
{Q_3}_*\begin{pmatrix}J_m(-1)& O& 2^{m+1}s\bar{B}& O\\
                    O&J_m(-1)& O& 2^{m+1}sK_m\\
                    C_{11}& O&J_m(-1)^T& O\\
                    O& C_{22}&O&J_m(-1)^T
        \end{pmatrix}
=\begin{pmatrix}J_m(-1)& O& s\bar{B}& O\\
                    O&J_m(-1)& O& sK_m\\
                    2^{m+1}{C}_{11}& O&J_m(-1)^T& O\\
                    O& 2^{m+1}{C}_{22}&O&J_m(-1)^T
        \end{pmatrix}
$$
Therefore, letting
$$
R_1=Q_3Q_2Q_1,
$$
we have
$$
{R_1}_*(M^2)=\begin{pmatrix}J_m(-1)&\bar{B}\\2^{m+1}C_{11}&J_m(-1)^T\end{pmatrix}
    \diamond
    \begin{pmatrix}J_m(-1)&sK_m\\2^{m+1}C_{22}&J_m(-1)^T\end{pmatrix}
$$

On the other hand, by \eqref{relation.C.C22},
we have $rank(C_{11})+rank(C_{22})=rank(CJ_{2m}(0))=rank(J_{2m}(0))=2m-1$;
and by Lemma \ref{Lm:simple.form.of.one.Jordan.block}, $rank(C_{22})=rank(2^{m+1}C_{22})=m-1$.
Thus $rank(C_{11})=m$.
Therefore, by Lemma \ref{Lm:simple.form.of.one.Jordan.block},
there exists an $R_2\in{\rm GL}_m(\mathbb{R})$ such that
$$
{R_2}_*\begin{pmatrix}J_m(-1)&\bar{B}\\2^{m+1}C_{11}&J_m(-1)^T\end{pmatrix}
=\begin{pmatrix}J_m(-1)&B_1\\-sK_m&J_m(-1)^T\end{pmatrix}
$$
for some symmetric matrix $B_1$.
Therefore, $R={\rm diag}(R_2,I_m)R_1$ is the required action matrix.
Note that
\begin{eqnarray}
\begin{pmatrix}O&-K_m\\K_m&O\end{pmatrix}
\begin{pmatrix}J_m(-1)&B_1\\-sK_m&J_m(-1)^T\end{pmatrix}
\begin{pmatrix}O&-K_m\\K_m&O\end{pmatrix}^{-1}
&=&\begin{pmatrix}K_mJ_m(-1)^TK_m&-K_m(-sK_m)K_m\\-K_mB_1K_m&K_mJ_m(-1)K_m\end{pmatrix}
\nonumber\\
&=&\begin{pmatrix}J_m(-1)&sK_m\\-K_mB_1K_m&J_m(-1)^T\end{pmatrix}.
\end{eqnarray}
Comparing the two Wonenburger matrices $\begin{pmatrix}J_m(-1)&sK_m\\2^{m+1}C_{22}&J_m(-1)^T\end{pmatrix}$ and $\begin{pmatrix}J_m(-1)&sK_m\\-K_mB_1K_m&J_m(-1)^T\end{pmatrix}$,
the equation in \eqref{constraints.of.Wonenburger.form} give
$$
sK_m\cdot2^{m+1}C_{22}=J_m(-1)^2-I=sK_m\cdot(-K_mB_1K_m).
$$
Since $sK_m$ is invertible, there must hold $2^{m+1}C_{22}=-K_mB_1K_m$.
Therefore, we have
$$
M^2\sim\begin{pmatrix}J_m(-1)&sK_m\\C_{2}&J_m(-1)^T\end{pmatrix}^{\diamond2},
$$
where $C_2=2^{m+1}C_{22}$.

\medskip
{\bf Case 2: $n=2l+1$ is odd.}

Let
$$
Q_2=\{q_{ij}\}_{(2l+1)\times(2l+1)},
$$
which entries satisfy
$$
q_{11}=1,
$$
and
$$
q_{i+1,2i+1}=q_{l+i+1,2i}=1,\quad\forall 1\le i\le l,
$$
and all other entries of $Q_2$ equal to $0$.
Then we have
$$
{Q_2}^{-1}={Q_2}^T.
$$
By a similar argument as in Case 1, we have
\begin{eqnarray}
    Q_2[Q_1(-I_{2l+1}+2J_{2l+1}(0)^2){Q_1}^{-1}]{Q_2}^{-1}&=&Q_2[-I_{2l+1}+J_{2l+1}(0)^2]Q_2^T
\nonumber\\
&=&{\rm diag}(J_{l+1}(-1),J_{l}(-1)),
\nonumber\\
Q_2[Q_1(2sJ_{2l+1}(0)K_{2l+1})Q_1^{T}]Q_2^T
&=&Q_2[2^{\frac{n}{2}+1}sJ_{2l+1}(0)K_{2l+1}]Q_2^T
\nonumber\\
&=&2^{l+\frac{3}{2}}\begin{pmatrix}
                        & & sK_l\\
                        & 0& &\\
                        sK_l& & &
                    \end{pmatrix}.
\nonumber
\end{eqnarray}
Furthermore, we define
$$
Q_3=2^{-\frac{2l+3}{4}}I_{2l+1},
$$
and let
$$
R=Q_3Q_2Q_1
$$
Then we obtain
\begin{eqnarray}
    R[-I_{2l+1}+2J_{2l+1}(0)^2]R^{-1}&=&{\rm diag}(J_{l+1}(-1),J_{l}(-1)),
    \nonumber\\
    R[2sJ_{2l+1}(0)K_{2l+1}]R^T&=&\begin{pmatrix}
                        & & sK_l\\
                        & 0& &\\
                        sK_l& & &
                    \end{pmatrix}.
\end{eqnarray}
\end{proof}

\section{Square root of symplectic matrices in $\Sp(2n)$}
\label{sec:4}

Suppose $X=\begin{pmatrix}X_{11}& X_{12}\cr X_{21}& X_{11}^T\end{pmatrix}$, as a Wonenburger matrix,
    is a real symplectic square root of $M=\begin{pmatrix}A&B\\C&A^T\end{pmatrix}$.
    Then we have
    \begin{eqnarray}
    \left(\begin{matrix}
        A & B\\
        C & A^T
    \end{matrix}\right)
    &=&\left(\begin{matrix}X_{11}& X_{12}\cr X_{21}& X_{11}^T\end{matrix}\right)^2
    \nonumber\\
    &=&\left(\begin{matrix}
    X_{11}^2+X_{12}X_{21} & X_{11}X_{12}+X_{12}X_{11}^T \\ X_{21}X_{11}+X_{11}^TX_{21} & (X_{11}^T)^2+X_{21}X_{12}
\end{matrix}\right)
    \nonumber\\
    &=&2\left(\begin{matrix}
    X_{11}^2 & X_{11}X_{12} \\ X_{21}X_{11} & (X_{11}^T)^2
\end{matrix}\right)-I_{2n}.,   
\end{eqnarray}
where we have used (\ref{constraints.of.Wonenburger.form}) in the last equality.
Consequently, the following must hold:
\begin{eqnarray}
    X_{11}^2&=&\frac{1}{2}(I_n+A),\label{eq.X_11}\\
    X_{11}X_{12}&=&\frac{1}{2}B,\label{eq.X_12}\\
    X_{21}X_{11}&=&\frac{1}{2}C.\label{eq.X_21}
\end{eqnarray}

Following an argument similar to that in Section 4 of \cite{FrM1},
the characteristic polynomial of $M$ is given by
$$
p_{M}(t)=\det(t^2I_n-2t(2X_{11}^2-I_n)+I_n)
=(4t)^n\det\left({\frac{(t+1)^2}{4t}}I_n-X_{11}^2\right).
$$
Then for any eigenvalue $\mu\in\sigma(X_{11})$,
$M$ has eigenvalues $t$ satisfying
\begin{equation}\label{Eq.of.t}
\frac{(t+1)^2}{4t}=\mu^2.
\end{equation}
If $t$ lies on the unit circle $\mathbb{U}\subset\mathbb{C}$, then $\mu^2=\frac{1}{4}(t+\bar{t}+2)\in(0,1)$.
If $t>1$ (or $0<t<1$ resp.), then $\mu^2=\frac{1}{4}(t+\frac{1}{t}+2)>1$.
If $t<-1$ (or $-1<t<0$ resp.), then $\mu^2=\frac{1}{4}(t+\frac{1}{t}+2)<-1$.
If $t\in\mathbb{C}\backslash(\mathbb{R}\cup\mathbb{U})$,
then $\mu^2=\frac{1}{4}(t+\frac{1}{t}+2)\in\mathbb{C}\backslash(\mathbb{R}\cup\mathbb{U})$.

The preceding relations show that the square-root problem is governed by the spectral behavior of the block $X_{11}$, and hence by the corresponding spectral position of $M$. In view of the decomposition results established in Section \ref{sec:3}, it is enough to treat the square-root problem on each spectral component separately. We shall divide the discussion into three cases. Subsection \ref{subsec:4.1} deals with the case in which $M$ has no eigenvalues on the negative real axis; in this situation the square root can be constructed directly. Subsection \ref{subsec:4.2} treats negative hyperbolic matrices, where the obstruction is essentially a parity and doubling phenomenon. Finally, Subsection \ref{subsec:4.3} is devoted to the degenerate case in which $-1$ is the only eigenvalue, whose analysis requires the refined normal forms obtained in the previous section.

\subsection{The case with no eigenvalues on the negative real axis}
\label{subsec:4.1}

When the matrix only possess positive eigenvalues, we have the following lemma:
\begin{lemma}\label{Lm:real.sqrt}
    Suppose $A\in\mathbb{R}^{n\times n}$ possesses only
    unique eigenvalues $\lambda$ with multiplicity $n$.
    If $\lambda>0$, then $A$ has at least one real square root.
    Furthermore, the real square root can be written as a polynomial function of $A$.
\end{lemma}
\begin{proof}
    Without loss of generality, we suppose
    $$
    A={\rm diag}(J_{n_1}(\lambda),\ldots,J_{n_k}(\lambda)),
    $$
    where $J_{n_i}(\lambda),i=1,\ldots,k$ are Jordan blocks with
    eigenvalues $\lambda$.
    Let $A_1=\lambda I_n$, then $A_2=A-A_1$ is a nilpotent matrix.

    We define
    $$
    X=\sqrt{\lambda}\sum_{i=0}^{+\infty}\left(\begin{matrix}\frac{1}{2}\\i\end{matrix}\right)(\frac{A_2}{{\lambda}})^i
    =\sqrt{\lambda}\sum_{i=0}^{n}\left(\begin{matrix}\frac{1}{2}\\i\end{matrix}\right)(\frac{A-\lambda I_n}{{\lambda}})^i,
    $$
    where in the second equality, we have used the nilpotency of $A_2=A-\lambda I_n$.
    Then $X$ is a polynomial function of $A$, and meanwhile a real square root of $A$.
\end{proof}

Now we can give 

\begin{proof}[The proof of Theorem A]
By the characteristic-polynomial identity
\[
p_M(t)=\det(t^2I_n-2tA+I_n),
\]
a number $t$ is an eigenvalue of $M$ only if
$\mu=\frac12\left(t+t^{-1}\right)$
is an eigenvalue of $A$. 
Conversely, if $\mu\in\sigma(A)$, then the corresponding eigenvalues of $M$ are the roots of $t^2-2\mu t+1=0$.
It follows that $M$ has no eigenvalues on the negative real axis if and only if
\[
\sigma(A)\cap(-\infty,-1]=\emptyset.
\]
Indeed, a negative real eigenvalue $t<0$ of $M$ gives
$\frac12(t+t^{-1})\leq -1$,
whereas each $\mu\leq -1$ gives negative real roots of $t^2-2\mu t+1=0$.

Set
\[
S=\frac12(I_n+A).
\]
Then
\[
\sigma(S)\cap(-\infty,0]=\emptyset .
\]
Hence the principal square root $S^{1/2}$ is well defined and real. 
Moreover, as a primary matrix function, $S^{1/2}$ can be represented by a real polynomial in $S$, and therefore also by a real polynomial in $A$; see Higham (\cite{Hig2}, Chapters 1 and 6). 
Put
\[
X_{11}=S^{1/2}=f(A),\quad 
X_{12}=\frac{1}{2}X_{11}^{-1}B,\quad 
X_{21}=\frac{1}{2}CX_{11}^{-1},
\]
where $f(A)$ is some polynomial function of $A$.
Then $X_{11}$ commutes with $A$, and from the Wonenburger relations
$AB=BA^T,A^TC=CA$,
we also have
\[
X_{11}B=BX_{11}^T,\qquad X_{11}^TC=CX_{11}.
\]
We need to check the symmetry of $X_{12}$. Indeed, we have
$$
X_{12}^T=\frac{1}{2}B^T(f(A)^{-1})^T=\frac{1}{2}f(A)^{-1}B=X_{12},
$$
where we have used $B=B^T$ and $AB=BA^T$ in the second equality.
By the similar arguments, we have
$$
X_{21}^T=X_{21},\quad
X_{11}X_{12}=X_{12}X_{11}^T,\quad
X_{11}^TX_{21}=X_{21}X_{11}.
$$
Moreover, we have
\begin{eqnarray}
X_{11}^2-X_{12}X_{21}
&=&\frac{1}{2}(I+A)-\frac{1}{2}X_{11}^{-1}B\cdot\frac{1}{2}CX_{11}^{-1}
\nonumber\\
&=&\frac{1}{2}(I+A)-\frac{1}{4}X_{11}^{-1}(A^2-I)X_{11}^{-1}
\nonumber\\
&=&\frac{1}{2}(I+A)-\frac{1}{4}(X_{11}^{-1})^2(A^2-I)
\nonumber\\
&=&\frac{1}{2}(I+A)-\frac{1}{2}(A+I)^{-1}(A^2-I)
\nonumber\\
&=&\frac{1}{2}(I+A)-\frac{1}{2}(A-I) \nonumber\\
&=&I,
\end{eqnarray}
where in the third equality, we used the commutativity of $A$ and $X_{11}=f(A)$.
Therefore, $X=
\begin{pmatrix}
X_{11} & X_{12}\\
X_{21} & X_{11}^T
\end{pmatrix}$ exactly satisfies (\ref{constraints.of.Wonenburger.form}),
which is a real symplectic matrix.

At last, we have
\[
X^2=
\begin{pmatrix}
X_{11} & X_{12}\\
X_{21} & X_{11}^T
\end{pmatrix}^2
    =2\left(\begin{matrix}
    X_{11}^2 & X_{11}X_{12} \\ X_{21}X_{11} & (X_{11}^T)^2
\end{matrix}\right)-I_{2n}
=
\begin{pmatrix}
A & B\\
C & A^T
\end{pmatrix}
=M.
\]
Thus $M$ admits a real symplectic square root.
\end{proof}

\subsection{The Negative Hyperbolic Case}
\label{subsec:4.2}

We turn to the second case, where the symplectic matrix possesses negative real eigenvalues.
In this case, not every matrix admits a real symplectic square root.

Now we give 

\begin{proof}[The proof of Theorem B]
We first prove the sufficient condition.
Suppose $n=2n_1$, and there exists $A_1\in\mathbb{R}^{n_1\times n_1}$ such that
$$
A\sim {\rm diag}(A_1,A_1).
$$
Since $M$ only possesses eigenvalues $\lambda,\frac{1}{\lambda}$, 
$A$, and then $A_1$ only have eigenvalue $\frac{1}{2}(\lambda+\frac{1}{\lambda})$.
Then $-\frac{I+A_1}{2}$ possesses positive eigenvalue
$\frac{(\lambda-1)^2}{-2\lambda}>0$.
By Lemma \ref{Lm:real.sqrt},
suppose $Y$ be a real square root of
$-\frac{I+A_1}{2}$.
Define
$$
X_{11}=\begin{pmatrix}
    O & -Y\\ Y & O
\end{pmatrix}.
$$
Then we have
$$
X_{11}^2=\begin{pmatrix}
    -Y^2 & O\\ O & -Y^2
\end{pmatrix}
=\begin{pmatrix}
    \frac{I+A_1}{2} & O\\ O & \frac{I+A_1}{2}
\end{pmatrix}
=\frac{I+A}{2}.
$$
The remaining argument is similar as Case (1) in the proof of
Theorem A.

Conversely, suppose $\lambda=-\mu^2$ for some $\mu>0$, 
and $X=\begin{pmatrix}X_{11}&X_{12}\\X_{21}&X_{22}^T\end{pmatrix}$ be a real symplectic square root of $M$.
Then we have $\sigma(X)=\{\pm\mu\sqrt{-1},\pm\frac{1}{\mu}\sqrt{-1}\}$,
each with the same multiplicity.
Furthermore, we have
$\sigma(X_{11})=\{\pm\frac{1}{2}(\mu-\frac{1}{\mu})\sqrt{-1}\}$.
Thus, $n$ is even, and so we may assume $n=2n_1$.

There exists $P\in {\rm GL}_n(\mathbb{R})$ and $m\in\mathbb{N}$
such that
$$
P^{-1}X_{11}P={\rm diag}(J_{2k_1},J_{2k_2},\ldots,J_{2k_m}),
$$
where
\begin{eqnarray}
J_{2k_i}:=J_{2k_i}(\hat\mu,\frac{\pi}{2})=
\begin{pmatrix}
    \hat\mu R(\frac{\pi}{2})& I_2& 0& \ldots& 0& 0\cr
    0& \hat\mu R(\frac{\pi}{2})& I_2& \ldots& 0& 0\cr
    \cdot& \cdot& \cdot& \ldots& \cdot& \cdot\cr
    \cdot& \cdot& \cdot& \ldots& \cdot& \cdot\cr
    0& 0& 0& \ldots& \hat\mu R(\frac{\pi}{2})& I_2\cr
    0& 0& 0& \ldots& 0& \hat\mu R(\frac{\pi}{2})
\end{pmatrix},
\end{eqnarray}
where $\hat\mu=\frac{1}{2}|\mu-\frac{1}{\mu}|$.
Then we have
$$
P^{-1}X_{11}^2P={\rm diag}(J_{2k_1}^2,J_{2k_2}^2,\ldots,J_{2k_m}^2).
$$
Furthermore, for $1\le i\le m$,
we obtain
\begin{eqnarray}
J_{2k_i}^2&=&
\begin{pmatrix}
    \hat\mu R(\frac{\pi}{2})& I_2& 0& \ldots& 0& 0\cr
    0& \hat\mu R(\frac{\pi}{2})& I_2& \ldots& 0& 0\cr
    \cdot& \cdot& \cdot& \ldots& \cdot& \cdot\cr
    \cdot& \cdot& \cdot& \ldots& \cdot& \cdot\cr
    0& 0& 0& \ldots& \hat\mu R(\frac{\pi}{2})& I_2\cr
    0& 0& 0& \ldots& 0& \hat\mu R(\frac{\pi}{2})
\end{pmatrix}^2
\nonumber\\
&=&
\begin{pmatrix}
    -\hat\mu^2 I_2& 2\hat\mu J_2& I_2& 0& \ldots& 0& 0\cr
    0& \hat\mu^2 I_2& 2\hat\mu J_2& I_2& \ldots& 0& 0\cr
    \cdot& \cdot& \cdot& \cdot& \ldots& \cdot& \cdot\cr
    \cdot& \cdot& \cdot& \cdot& \ldots& \cdot& \cdot\cr
    0& 0& 0& 0& \ldots& 2\hat\mu J_2& I_2\cr
    0& 0& 0& 0&\ldots& -\hat\mu^2 I_2& 2\hat\mu J_2\cr
    0& 0& 0& 0&\ldots& 0& -\hat\mu^2 I_2
\end{pmatrix}.    \nonumber
\end{eqnarray}
Note that $\dim\ker(J_{2k_i}(\hat\mu,\frac{\pi}{2})^2+\hat\mu^2 I_{2k_i})=2$.
Moreover, we have
$\Big(J_{2k_i}(\hat\mu,\frac{\pi}{2})^2+\hat\mu^2 I_{2k_i}\Big)^{k-1}\ne0$ and
$\Big(J_{2k_i}(\hat\mu,\frac{\pi}{2})^2+\hat\mu^2I_{2k_i}\Big)^{k}=0$.
Then the Jordan form of $J_{2k_i}(\hat\mu,\frac{\pi}{2})$ has only two blocks,
each with size $k_i\times k_i$.
Hence, there exists $Q_i\in GL(2k_i)$ such that
$$
Q_i^{-1}J_{2k_i}(\hat\mu,\frac{\pi}{2})Q_i= diag(J_{k_i}(-\hat\mu^2),J_{k_i}(-\hat\mu^2)).
$$
Then $A_1=2{\rm diag}(J_{k_1}(-\hat\mu^2),J_{k_2}(-\hat\mu^2),\ldots,J_{k_m}(-\hat\mu^2))-I_n$
is the required matrix.
\end{proof}

\subsection{The Degenerate Case: Eigenvalue $-1$}
\label{subsec:4.3}

Finally, we address the third and most intricate case, where the symplectic matrix has $-1$ as its only eigenvalue. This degenerate scenario is characterized by the presence of nontrivial nilpotent structures in the Jordan normal form of its top-left block $A$. The analysis consequently requires a more refined decomposition theory to account for this added complexity.


Having established the necessary decompositions in Section 3, we now give

\begin{proof}[The Proof of Theorem C]
Suppose $X$ is a real symplectic square root
    of $M$.
    Then we have $\sigma(X)=\{\pm\sqrt{-1}\}$,
    each with multiplicity $n$.
    By Theorem \ref{Thm:decomposition.of.eig.i},
    there exists $Q\in\Sp(2n)$ such that
\begin{equation}
    QXQ^{-1}
    =(M_1
    \diamond\ldots\diamond M_p)\diamond(N_1\diamond\ldots\diamond N_q),
\end{equation}
where 
$$
M_i=\begin{pmatrix}J_{2m_i}(0)&s_i{K}_{2m_i}\\**&J_{2m_i}(0)^T\end{pmatrix},
$$
for some $m_i\in\mathbb{N},s_i\in\{\pm1\},\;1\le i\le p$,
and
$$
N_j=\begin{pmatrix}J_{2l_j+1}(0)&t_j{K}_{2l_j+1}\\**&J_{2l_j+1}(0)^T\end{pmatrix},
$$
for some $l_j\in\mathbb{N}\cup\{0\},t_j\in\{\pm1\},\;1\le j\le q$.

On the other hand, we have
$$
QMQ^{-1}=(QXQ^{-1})^2=(M_1^2\diamond\ldots\diamond M_p^2)\diamond(N_1^2\diamond\ldots\diamond N_q^2).
$$
By Lemma \ref{Lm:square.of.eig.i}, for each $i\in\{1,2\ldots,p\}$ and
$j\in\{1,2\ldots,q\}$,
we have
$$
M_i^2\sim\begin{pmatrix}J_{m_i}(-1)&s_i{K}_{m_i}\\D_i&J_{m_i}(-1)^T\end{pmatrix}^{\diamond2},
$$
for some $D_i\in{\mathbb{R}}^{m_i\times m_i}$,
and
$$
N_j^2\sim\begin{pmatrix}A_j&B_j\\C_j&A_j^T\end{pmatrix}
$$
where $A_j$ and $B_j$ are given by (\ref{Aj.Bj}).

Conversely, if we have the decomposition (\ref{existence.condition.of.eig.-1}),
then by reversing
the proof of Lemma \ref{Lm:square.of.eig.i},
we can construct a real symplectic square root of $M$.
\end{proof}

\section*{Acknowledgements}

The author would like to thank Prof. Urs Frauenfelder for many helpful suggestions and remarks.

\end{document}